\documentclass[11pt]{amsart}
\usepackage{amssymb, amsmath, amscd, eucal, rotating}
\usepackage{tikz-cd}

\usepackage[foot]{amsaddr}

\usepackage{orcidlink}

\input xy
 \xyoption{all}

\numberwithin{equation}{section}

\theoremstyle{plain}

\newtheorem{proposition}{Proposition}[section]
\newtheorem{theorem}[proposition]{Theorem}		
\newtheorem{corollary}[proposition]{Corollary}
\newtheorem{lemma}[proposition]{Lemma}

\theoremstyle{definition}

\newtheorem{definition}[proposition]{Definition}
\newtheorem{remark}[proposition]{Remark}

\newcommand{\SL}{\mathsf{SL}}

\newcommand{\GL}{\mathsf{GL}}
\newcommand{\SU}{\mathsf{SU}}

\newcommand{\U}{\mathsf{U}}

\DeclareMathOperator{\Sec}{Sec}
\DeclareMathOperator{\Jac}{Jac}
\DeclareMathOperator{\Gr}{Gr}
\DeclareMathOperator{\End}{End}
\DeclareMathOperator{\Sing}{Sing}

\DeclareMathOperator{\id}{id}

\DeclareMathOperator{\tr}{tr}

\DeclareMathOperator{\rank}{rank}
\DeclareMathOperator{\vol}{vol}
\DeclareMathOperator{\spanC}{span_\mathbb{C}}

\DeclareMathOperator{\Res}{Res}

\begin{document}

% Topmatter

\title{Flow lines on the moduli space of rank $2$ twisted Higgs bundles}	

\author[Graeme Wilkin]{Graeme Wilkin \, \orcidlink{0000-0002-1504-7720}}

\thanks{}


\email{graeme.wilkin@york.ac.uk, ORCID: 0000-0002-1504-7720}
\address{Department of Mathematics,
University of York, 
York YO10 5DD, United Kingdom}

\date{\today}

\begin{abstract} 
This paper studies the gradient flow lines for the $L^2$ norm square of the Higgs field defined on the moduli space of semistable rank $2$ Higgs bundles twisted by a line bundle of positive degree over a compact Riemann surface $X$. The main result is that these spaces of flow lines have an algebro-geometric classification in terms of secant varieties for different embeddings of $X$ into the projectivisation of the negative eigenspace of the Hessian at a critical point. The Morse-theoretic compactification of spaces of flow lines given by adding broken flow lines then has a natural algebraic interpretation via a projection to Bertram's resolution of secant varieties.
\end{abstract}

% End Topmatter

%%%%%%%%%%%%%%%%%%%%
% Disclaimer
%%%%%%%%%%%%%%%%%%%%
%\begin{center}
%\framebox{
%{\Large\bf DRAFT (\today): DO NOT DISTRIBUTE.}}
%\end{center}

\maketitle

\thispagestyle{empty}

%\newpage

\baselineskip=16pt
\setcounter{footnote}{0}

%%%%%%%%%%%%%%%%%%%%%%%%%%%%%%%%%%%%%%%%%%%%%%%%

\section{Introduction}

The moduli space of Higgs bundles over a compact Riemann surface admits a natural Morse-Bott function, given by the square of the $L^2$ norm of the Higgs field. This has been instrumental in efforts to understand more about the topology of this space, with an enormous amount of activity beginning with the original work of Hitchin \cite{Hitchin87}. There has been much recent interest in gradient flow lines of this function and their connection with Geometric Langlands, which has focused on understanding the \emph{very stable} and \emph{wobbly} bundles (see \cite{DonagiPantev09}, \cite{FrancoGothenPeonNieto21}, \cite{FrancoGothenOliveiraPeonNieto24}, \cite{HauselHitchin22}, \cite{Laumon88}, \cite{PeonNieto24}).

Motivated by the geometric description of Yang-Mills-Higgs flow lines in \cite{Wilkin20}, the goal of this paper is to give a concrete description of the space of flow lines connecting two critical sets for the function $\| \phi \|_{L^2}^2$. In general, for any Morse-Bott function, the unstable set of a critical set is stratified by the types of the critical sets that can appear as downwards limits of the flow. The main result of this paper is that, for rank $2$ Higgs bundles, this stratification has a geometric interpretation in terms of secant varieties for different embeddings of the underlying Riemann surface into the projectivisation of the negative eigenbundle of the Hessian at each critical point.  Moreover, this geometric description of the flow lines also leads to a simple proof that the function $\| \phi \|_{L^2}^2$ is in fact Morse-Bott-Smale, and therefore one can use the methods of \cite{AustinBraam95} to construct a Morse complex in which the cup product is determined by the topology of the spaces of flow lines.

For rank $2$ Higgs bundles twisted by a line bundle $M$, the nonminimal critical points of $\| \phi \|_{L^2}^2$ have the form $[L_1 \oplus L_2, \phi]$, where $L_1, L_2$ are line bundles with $\deg L_1 > \deg L_2$ and $\phi \in H^0(L_1^*L_2 \otimes M)$. With respect to the downwards gradient flow of $\| \phi \|_{L^2}^2$, the unstable manifold of such a critical point is homeomorphic to $H^1(L_1^*L_2)$, which parametrises extensions $0 \rightarrow L_2 \rightarrow E \rightarrow L_1 \rightarrow 0$. Since $\deg L_1 > \deg L_2$, then there is a canonical embedding $X \hookrightarrow \mathbb{P}H^1(L_1^*L_2)$, and the first main result is that the limit of the downwards flow with initial condition in the unstable manifold of the critical point $[L_1 \oplus L_2, \phi]$ is determined by the secant varieties of $X \hookrightarrow \mathbb{P}H^1(L_1^*L_2)$.

\begin{theorem}[Theorem \ref{thm:space-unbroken-flow-lines}]\label{thm:space-unbroken-flow-lines-intro}
Fix $\rank(E) = 2$, $\deg E = 0 \, \, \text{or} \, \, 1$ and let $C_\ell, C_u$ be two critical sets indexed by $0 \leq \ell < u \leq \frac{1}{2} \deg E + g-1$. Then the space $\mathcal{L}_\ell^u$ of flow lines between $C_\ell$ and $C_u$ is a circle bundle over the $(u-\ell)^{th}$ global secant variety $\mathcal{P}_\ell^u$ from Definition \ref{def:secant-bundle}, where the fibres are the orbits of the $S^1$ action $e^{i \theta} \cdot [E, \phi] = [E, e^{i \theta} \phi]$ on $\mathcal{M}_{Higgs}^{ss}(E)$.
\end{theorem}

\begin{remark}
The notation $\mathcal{L}_\ell^u$ is used for the space of flow lines and $\mathcal{F}_\ell^u$ is used for the space of points that flow up to $C_u$ and down to $C_\ell$ (cf. \eqref{eqn:R-bundle-flow-lines} and \eqref{eqn:unbroken-flow-lines}). The two spaces are related by $\mathcal{L}_\ell^u := \mathcal{F}_\ell^u / \mathbb{R}$, where $\mathbb{R}$ acts by time translation along a flow line.
\end{remark}

Therefore we have a parametrisation of the \emph{unbroken} flow lines connecting two critical sets, however to construct a Morse-Bott complex on the moduli space one needs to prove that the function satisfies the stronger Morse-Bott-Smale condition. Theorem \ref{thm:space-unbroken-flow-lines-intro} leads to a simple proof of the Morse-Bott-Smale property (see Proposition \ref{prop:morse-bott-smale}), which has the following consequences for the cup product on the Morse complex.

There are canonical projections $\pi_\ell : \mathcal{F}_\ell^u \rightarrow C_\ell$ and $\pi_u : \mathcal{F}_\ell^u \rightarrow C_u$ given by taking the limit of the flow as $t \rightarrow \pm \infty$. Similarly, there is canonical projection $p_u : \mathcal{P}_\ell^u \rightarrow C_u$ (from Definition \ref{def:secant-bundle}) as well as a projection $p_\ell : \mathcal{P}_\ell^u \rightarrow C_\ell$ taking a point in a secant plane to the limiting Higgs pair from Section \ref{subsec:explicit-gauge}.

\begin{equation}\label{eqn:projections-factor}
\begin{tikzcd}
 & \arrow[bend right=10]{ddl}[swap]{\pi_\ell} \mathcal{F}_\ell^u \arrow{d}{g} \arrow[bend left=10]{ddr}{\pi_u} \\
 & \arrow{dl}{p_\ell} \mathcal{P}_\ell^u \arrow{dr}[swap]{p_u} \\
C_\ell & & C_u 
\end{tikzcd}
\end{equation}

Let $\eta \in H^*(C_\ell)$ and $\omega \in H^*(\mathcal{M}_{Higgs}^{ss}(E))$. The class $\omega$ restricts to a cohomology class on $\mathcal{F}_\ell^u \subset \mathcal{M}_{Higgs}^{ss}(E)$, which we also denote by $\omega$. The cup product in the Morse-Bott complex (see \cite[Sec. 3.5]{AustinBraam95}) for $\mathcal{M}_{Higgs}^{ss}(E)$ is given by
\begin{equation*}
c_\eta(\omega) = (\pi_u)_* \left( \pi_\ell^*(\eta) \smallsmile \omega \right) .
\end{equation*}
The previous theorem shows that this factors through the global secant variety via the diagram \eqref{eqn:projections-factor}. Therefore, using the results of \cite{AustinBraam95}, the cup product on the Morse complex can be expressed entirely in terms of the homomorphisms $p_\ell^*$ and $(p_u)_*$
\begin{equation}
c_\eta(\omega) = (p_u)_* g_* \left( g^* p_\ell^*(\eta) \smallsmile \omega \right) =  (p_u)_* \left( p_\ell^*(\eta) \smallsmile g_*(\omega) \right) ,
\end{equation}
and therefore the cup product can be computed using the topology of secant varieties.

Finally, it is natural to construct the Morse-Bott-Smale compactification of the space of flow lines given by attaching spaces of broken flow lines (this is explained in detail by Austin and Braam \cite{AustinBraam95}). The next theorem shows that this compactification of the space of flow lines has an algebro-geometric interpretation in terms of the resolution of secant varieties constructed by Bertram \cite{Bertram92}.

\begin{theorem}[Theorem \ref{thm:Morse-resolution}]\label{thm:Morse-resolution-intro}
Let $P_{Morse} : \widetilde{\mathcal{L}_\ell^u} \rightarrow \overline{\mathcal{L}_\ell^u}$ be the Morse resolution associated to the compactification of broken flow lines from \eqref{eqn:Morse-resolution} and let $P_{Sec} : \widetilde{\mathcal{P}_\ell^u} \rightarrow \overline{\mathcal{P}_\ell^u}$ be the resolution of secant varieties defined by Bertram \cite{Bertram92} (cf. \eqref{eqn:secant-resolution}). Then the map that takes a broken flow line to the corresponding chain of points in secant planes from Definition \ref{def:resolution-map} makes the following diagram commute.
\begin{equation}\label{eqn:relate-resolutions-intro}
\begin{tikzcd}[column sep=1.5cm]
\widetilde{\mathcal{L}_\ell^u} \arrow{r}{\text{Def. \ref{def:resolution-map}}} \arrow{d}[swap]{P_{Morse}} & \widetilde{\mathcal{P}_\ell^u} \arrow{d}{P_{Sec}} \\
\overline{\mathcal{L}_\ell^u} \arrow{r}{\text{Prop. \ref{prop:flow-secant-closure}}} & \overline{\mathcal{P}_\ell^u} 
\end{tikzcd}
\end{equation}
\end{theorem}

{\bf Organisation of the paper.} Section \ref{sec:background} contains the background material and notational conventions used throughout the paper. The secant varieties of the Riemann surface in the unstable manifold of a critical point are constructed in Section \ref{sec:secant-varieties}, which leads to the classification of the unbroken flow lines (Theorem \ref{thm:space-unbroken-flow-lines}) in Section \ref{sec:flow-line-classification}, and in turn a proof that $\| \phi \|_{L^2}^2$ is Morse-Bott-Smale in Section \ref{sec:Morse-Bott-Smale}. Finally, Section \ref{sec:compactification} contains the details of the compactification of the space of flow lines and its relation with the resolution of secant varieties (Theorem \ref{thm:Morse-resolution}).

{\bf Acknowledgements.} The author gratefully acknowledges support from the Simons Center for Geometry and Physics, Stony Brook University and the organisers of the summer workshop on Moduli which provided the motivation for this paper, as well as Tamas Hausel, Ana Pe\'on Nieto and Paul Feehan for useful discussions.

\section{Background and notational conventions}\label{sec:background}

In this section we recall the relevant results and set the notation for the remainder of the paper. Unless otherwise noted, the material is well understood and can be found in \cite{Hitchin87} or \cite{Simpson92}.

Let $X$ be a compact Riemann surface of genus $g \geq 2$ and let $E \rightarrow X$ be a rank $2$ complex vector bundle of degree $d = 0 \, \, \text{or} \, \,  1$. Fix a smooth Riemannian metric on $X$ and a smooth Hermitian metric on $E$, and let $\mathcal{A}^{0,1}$ denote the space of holomorphic structures on $E$. The complex gauge group is denoted $\mathcal{G}^\mathbb{C}$ and the unitary gauge group associated to the Hermitian metric is denoted $\mathcal{G}$.

One can also consider the determinant map $\det : \mathcal{A}^{0,1} \rightarrow \Jac_d(X)$ and for a fixed $\xi \in \Jac_d(X)$ the subsets $\mathcal{A}_\xi^{0,1} := \det^{-1}(\xi)$ and $\End_0(E) := \{ u \in \End(E) \, \mid \, \tr(u) = 0 \}$. Let $M \rightarrow X$ be a line bundle with $\deg M > 0$. For the two cases $G = \GL(2, \mathbb{C})$ or $G = \SL(2, \mathbb{C})$, the \emph{space of $G$-Higgs pairs twisted by $M$} is denoted
\begin{align*}
\mathcal{B} & = \{ (\bar{\partial}_A, \phi) \in \mathcal{A}^{0,1} \times \Omega^0(\End(E) \otimes M) \, \mid \, \bar{\partial}_A \phi = 0 \} \quad \text{($G = \GL(2, \mathbb{C})$)} \\
\mathcal{B}_\xi & = \{ (\bar{\partial}_A, \phi) \in \mathcal{A}_\xi^{0,1} \times \Omega^0(\End_0(E) \otimes M) \, \mid \, \bar{\partial}_A \phi = 0 \} \quad \text{($G = \SL(2, \mathbb{C})$).}
\end{align*}

Many of the constructions for $G = \SL(2, \mathbb{C})$ are the same as those for $G = \GL(2, \mathbb{C})$. To avoid notational complexity, from now on we will use the notation for $G = \GL(2, \mathbb{C})$ and only distinguish between the two cases when necessary; for example when specifying the fixed point sets of the $\mathbb{C}^*$ action.

The open subset of stable (resp. semistable) Higgs pairs is denoted $\mathcal{B}^{st}$ (resp. $\mathcal{B}^{ss}$) and the moduli space of stable (resp. semistable) Higgs bundles on $E$ is denoted 
\begin{equation*}
\mathcal{M}_{Higgs}^{st}(E) := \mathcal{B}^{st} / \mathcal{G}^\mathbb{C} \quad \text{ (resp. $\mathcal{M}_{Higgs}^{ss}(E) := \mathcal{B}^{ss} / \negthinspace / \mathcal{G}^\mathbb{C}$)} .
\end{equation*}
If $\deg E$ and $\rank(E)$ are coprime then $\mathcal{B}^{st} = \mathcal{B}^{ss}$ and the moduli space is a smooth manifold.

For each $(\bar{\partial}_A, \phi) \in \mathcal{B}^{ss}$, the associated equivalence class in $\mathcal{M}_{Higgs}^{ss}(E)$ is denoted by $[\bar{\partial}_A, \phi]$. When the holomorphic bundle is a direct sum $L_1 \oplus L_2$ or an extension $0 \rightarrow L_2 \rightarrow E \rightarrow L_1 \rightarrow 0$ of line bundles, then it is more convenient to use the notation $[L_1 \oplus L_2, \phi]$ or $[E, \phi]$.

With respect to the fixed Hermitian metric, each holomorphic structure $\bar{\partial}_A$ has an associated Chern connection denoted $d_A$ with curvature $F_A$. \emph{Hitchin's equations} for the Higgs pair $(\bar{\partial}_A, \phi)$ are
\begin{equation}\label{eqn:hitchin-equations}
*(F_A + [\phi, \phi^*]) = \lambda \cdot \id, \quad \text{where $\lambda = - \frac{2 \pi i \deg(E)}{\vol(X) \rank(E)}$} .
\end{equation}
The Hitchin-Kobayashi correspondence of Hitchin \cite{Hitchin87} and Simpson \cite{Simpson88} shows that
\begin{equation}\label{eqn:hitchin-simpson}
\mathcal{M}_{Higgs}^{ss}(E) \cong \{ (\bar{\partial}_A, \phi) \in \mathcal{B} \, \mid \, \text{$(\bar{\partial}_A, \phi)$ satisfy \eqref{eqn:hitchin-equations}}  \} / \mathcal{G} .
\end{equation}

\subsection{Properties of the energy function}\label{subsec:properties-energy-function}

The function $\| \phi \|_{L^2}^2 : \mathcal{B} \rightarrow \mathbb{R}$ is $\mathcal{G}$-invariant, and so \eqref{eqn:hitchin-simpson}  shows that the restriction to the solutions of \eqref{eqn:hitchin-equations} descends to a well-defined function
\begin{equation*}
f := \| \phi \|_{L^2}^2 : \mathcal{M}_{Higgs}^{ss}(E) \rightarrow \mathbb{R} ,
\end{equation*}
which is the moment map associated to the circle action $e^{i \theta} \cdot [\bar{\partial}_A, \phi] = [\bar{\partial}_A, e^{i\theta} \phi]$ (cf. \cite{Hitchin87}). The general result of Frankel \cite{Frankel59} shows that (when the moduli space is smooth) $f$ is a perfect Morse-Bott function.

This action extends to a $\mathbb{C}^*$ action $e^{u} \cdot [\bar{\partial}_A, \phi] = [\bar{\partial}_A, e^{u} \phi]$ for $u \in \mathbb{C}$. The gradient flow lines of $f$ on $\mathcal{M}_{Higgs}^{ss}(E)$ are generated by the subgroup $\mathbb{R}_{>0} \subset \mathbb{C}^*$, for which the action is 
\begin{equation}\label{eqn:grad-flow-generation}
e^t \cdot [\bar{\partial}_A, \phi] = [\bar{\partial}_A, e^{t} \phi], \quad t \in \mathbb{R} .
\end{equation}

The minimum $f^{-1}(0)$ corresponds to the subset of semistable Higgs pairs with zero Higgs field, which is the moduli space of semistable holomorphic bundles $\mathcal{M}^{ss}(E) \hookrightarrow \mathcal{M}_{Higgs}^{ss}(E)$. The nonminimal critical points of $f$ correspond to fixed points of the $\mathbb{C}^*$ action, which have a well-understood classification in terms of variations of Hodge structure (cf. \cite{Simpson92}). In general, a Higgs pair $[(E, \phi)]$ is a fixed point of the $\mathbb{C}^*$ action if and only if 
\begin{enumerate}

\item the bundle decomposes as a direct sum $E \cong F_1 \oplus \cdots \oplus F_n$,

\item for each $j = 1, \ldots, n-1$ the Higgs field satisfies $\phi(F_j) \subset F_{j+1} \otimes M$, and

\item the resulting Higgs pair is semistable.

\end{enumerate}
The equation \eqref{eqn:hitchin-equations} then determines the value of $f = \| \phi \|_{L^2}^2$ at each critical set. 

In the rank $2$ case, the minimum of $f$ occurs when the holomorphic bundle is semistable and the Higgs field is zero. At a nonminimal critical point, the holomorphic bundle is a direct sum of line bundles $E \cong L_1 \oplus L_2$ with $\deg L_2 < \deg L_1 \leq \deg L_2 + \deg M$ and $\phi \in H^0(L_1^* L_2 \otimes K)$. Let $d_1 = \deg L_1$, $d_2 = \deg L_2$ and $m = \deg M$. The critical values are ordered by the value of $d_1-d_2$, and each critical set is homeomorphic to $S^{d_2 - d_1+m}X \times J(X)$ ($G = \GL(2, \mathbb{C})$) and $\tilde{S}^{d_2-d_1+m}X$ ($G = \SL(2, \mathbb{C})$), where $\tilde{S}^d X$ denotes the $2^{2g}$ fold cover of $S^d X$ studied by Hitchin \cite[Prop. 7.1]{Hitchin87}.

In the sequel, the nonminimal critical set corresponding to $\deg L_1 = d$ will be denoted $C_d$ for each integer value of $d$ in the range $\frac{1}{2} \deg E < d \leq \frac{1}{2} \left( \deg E + \deg M \right)$. The minimum (corresponding to $\phi = 0$) will be denoted $C_0$. Using the fact that the gradient flow of $f$ is defined using the $\mathbb{R}_{>0}$ action \eqref{eqn:grad-flow-generation}, the stable and unstable sets of $C_d$ are defined by
\begin{align*}
W_d^+ & := \{ [\bar{\partial}_A, \phi] \in \mathcal{M}_{Higgs}^{ss}(E) \, \mid \, \lim_{\lambda \rightarrow 0} [\bar{\partial}_A, \lambda \phi] \in C_d \} \\
W_d^- & := \{ [\bar{\partial}_A, \phi] \in \mathcal{M}_{Higgs}^{ss}(E) \, \mid \, \lim_{\lambda \rightarrow \infty} [\bar{\partial}_A, \lambda \phi] \in C_d \} .
\end{align*}
The corresponding spaces with the critical sets removed are denoted
\begin{equation*}
W_{d,0}^+ := W_d^+ \setminus C_d, \quad W_{d,0}^- := W_d^- \setminus C_d .
\end{equation*}
If the degree and rank of $E$ are coprime then $W_d^+, W_d^-$ are manifolds. In general, the action of $\mathbb{R}_{>0} \subset \mathbb{C}^*$ still determines a well-defined integral curve, and therefore the spaces $W_d^+, W_d^-$ are still well-defined for $E$ of any degree and rank.

For any $\ell < u$, the space of Higgs pairs that flow down to $C_\ell$ and up to $C_u$ is denoted 
\begin{equation}\label{eqn:R-bundle-flow-lines}
\mathcal{F}_\ell^u := W_\ell^+ \cap W_u^-.
\end{equation}
The space of \emph{unbroken flow lines} between two critical sets is then given by dividing by the $\mathbb{R}_{>0}$ action generating the flow
\begin{equation}\label{eqn:unbroken-flow-lines}
\mathcal{L}_\ell^u := \mathcal{F}_\ell^u / \mathbb{R}_{>0} .
\end{equation}

\subsection{Morse strata for the gradient flow}

Recent work of Hausel and Hitchin \cite[Prop. 3.4 \& 3.11]{HauselHitchin22} gives a complete classification of the Morse strata for the upwards and downwards flow of $f$ in terms of the existence of filtrations of the underlying holomorphic bundle for which the Higgs field satisfies a compatibility condition. Their results are summarised in the following

\begin{proposition}[\cite{HauselHitchin22}]\label{prop:hausel-hitchin-general}
Let $[(E, \phi)], [(E', \phi')] \in \mathcal{M}_{Higgs}^{ss}$. Then
\begin{enumerate}

\item $\lim_{\lambda \rightarrow 0} [(E, \phi)] = [(E', \phi')]$ if and only if there exists a filtration by subbundles
\begin{equation*}
0 = E_0 \subset E_1 \subset \cdots \subset E_n = E
\end{equation*}
such that $\phi(E_k) \subset E_{k+1} \otimes K$ for each $k = 1, \ldots, n-1$, and the induced maps
\begin{equation*}
gr_0(\phi) : E_{k}/E_{k-1} \rightarrow E_{k+1} / E_{k}
\end{equation*}
satisfy $(E', \phi') \cong (E_1 / E_0 \oplus \cdots \oplus E_n / E_{n-1}, gr_0(\phi))$.

\item $\lim_{\lambda \rightarrow \infty} [(E, \phi)] = [(E', \phi')]$ if and only if there exists a filtration by subbundles
\begin{equation*}
0 = E_0 \subset E_1 \subset \cdots \subset E_n = E
\end{equation*}
such that $\phi(E_{k+1}) \subset E_k \otimes K$ for each $k = 0, \ldots, n-1$, and the induced maps
\begin{equation*}
gr_\infty(\phi) : E_{k+1}/E_{k} \rightarrow E_{k} / E_{k-1}
\end{equation*}
satisfy $(E', \phi') \cong (E_1 / E_0 \oplus \cdots \oplus E_n / E_{n-1}, gr_\infty(\phi))$.

\end{enumerate}

Moreover, the filtrations with the properties in (i) and (ii) are unique.
\end{proposition}

In the case of rank $2$, part (i) of the above result reduces to an earlier observation of Hitchin \cite{Hitchin87} that the Morse strata for the downwards flow correspond to the Harder-Narasimhan strata for the underlying subbundle (see also \cite{hauselthesis}). The same is true for $\U(2,1)$ and $\SU(2,1)$ Higgs bundles \cite{Gothen02}. For rank $3$ and higher this is no longer true and the stratification (studied in detail by Gothen and Zuniga-Rojas \cite{GothenZunigaRojas17} for the rank $3$ case) is much more intricate. 

In order to apply \cite[Prop. 3.4 \& 3.11]{HauselHitchin22} to a specific Higgs pair $[(\bar{\partial}_A, \phi)] \in \mathcal{M}_{Higgs}^{ss}(E)$, one needs to first find a filtration of the appropriate type. The results of Section \ref{sec:flow-line-classification} give a criterion for such filtrations to exist in the unstable set of each critical set. This criterion is geometric in nature and emphasises the role of the complex structure on the underlying Riemann surface $X$.

\subsection{The polystable locus in the rank $2$ case}\label{subsec:rank-2-remarks}

In general, the singularities of $\mathcal{M}_{Higgs}^{ss}(E)$ are contained in the locus where $[E, \phi]$ is strictly polystable, so that there is a direct sum
\begin{equation*}
(E, \phi) \cong (E_1, \phi_1) \oplus  \cdots \oplus (E_n, \phi_n)
\end{equation*}
of $n>1$ stable Higgs pairs of the same slope. When $\rank (E) = 2$ and $\deg(E) = 0$, then this can only occur if $(E, \phi)$ is a direct sum of two Higgs line bundles. In this case it was observed in \cite{DWW10} that the locus of strictly polystable bundles does not intersect any of the nonminimal critical sets. In particular, the unstable sets $W_d^-$ are all manifolds and all of the stable sets $W_d^+$ are manifolds when $d > 0$. Therefore, even though $\mathcal{M}_{Higgs}^{ss}(E)$ is singular, when $\rank(E) = 2$ it still makes sense to refer to $\| \phi \|_{L^2}^2$ as a perfect Morse-Bott function, or (after proving Proposition \ref{prop:morse-bott-smale}) a perfect Morse-Bott-Smale function, since it has a well-defined flow given by the $\mathbb{R}_{>0} \subset \mathbb{C}^*$ action, and the spaces of flow lines and the local structure around the \emph{nonminimal} critical sets satisfy these conditions.

\subsection{The unstable manifold in a neighbourhood of a critical set}

In the rank $2$ case studied in \cite{Hitchin87}, the critical sets consist of Higgs bundles for which the holomorphic bundle is a direct sum of line bundles $E \cong L_1 \oplus L_2$ (we will always use the convention that $\deg L_1 > \deg L_2$) and the Higgs field is $\phi \in H^0(L_1^* L_2 \otimes M) \setminus \{0\}$. The next result is contained in \cite{Hitchin87}, and we state it here in order to use it in the sequel.

\begin{lemma}\label{lem:classify-unstable-set}
In a neighbourhood of a critical point $y := [L_1 \oplus L_2, \phi]$, the unstable manifold $W_y^-$ of points that flow up to $y$ is given by equivalence classes of Higgs pairs for which the bundle is an extension
\begin{equation*}
0 \rightarrow L_2 \rightarrow E \rightarrow L_1 \rightarrow 0
\end{equation*}
and the Higgs field is $\phi \in H^0(L_1^* L_2 \otimes M) \subset H^0(\End(E) \otimes M)$. 

In particular, the unstable manifold is parametrised by the space of extensions $W_y^- \cong H^1(L_1^*L_2)$ and the Morse index at the critical point $y$ is then given by
\begin{equation}\label{eqn:Morse-index}
\lambda_y : = \dim_\mathbb{R} W_y^- = 2 h^1(L_1^*L_2) = 2g-2 - 2 \deg(L_1^*L_2) = 2g-2+2(\deg L_1 - \deg L_2) .
\end{equation}

For a given critical set $C_d$, the unstable manifold, denoted $W_d^-$, is a vector bundle over $C_d$ with fibre over the critical point $y := [L_1 \oplus L_2, \phi]$ given by $W_y^- \cong H^1(L_1^*L_2)$.
\end{lemma}

Since the critical values are isolated, for each critical set $C_d$ there exists $\varepsilon > 0$ such that there are no critical values in the interval $\left[ f(C_d) - \varepsilon, f(C_d) \right)$. Define
\begin{equation}\label{eqn:neg-sphere-bundle}
S_d^- := \{ z \in W_d^- \, \mid \, f(z) = f(C_d) - \varepsilon \} ,
\end{equation}
which is homeomorphic to a sphere bundle $S_d^- \rightarrow C_d$, for which the fibre over $x \in C_d$ is a sphere of dimension $\lambda_y - 1$.

\section{Secant varieties in the unstable manifold}\label{sec:secant-varieties}

Each unstable manifold $W_d^-$ is stratified by
\begin{equation}\label{eqn:unstable-stratification}
W_d^- = \bigcup_{0 \leq \ell < d} (\mathcal{F}_\ell^d \cap W_d^-) .
\end{equation}
In a neighbourhood of the critical set $C_d$, the unstable set $W_d^-$ is diffeomorphic to a vector bundle $V_d^- \rightarrow C_d$ with fibres $H^1(L_1^* L_2)$. Since the $\mathbb{C}^*$ action acts by scaling the extension classes in these fibres, and the gradient flow is $\mathbb{C}^*$ equivariant, then the stratification \eqref{eqn:unstable-stratification} descends to the projectivisation $\mathbb{P} V_d^-$. In the next section (Theorem \ref{thm:space-unbroken-flow-lines}) we will show that the strata have a geometric description in terms of certain secant varieties for the embedding $X \hookrightarrow \mathbb{P}H^1(L_1^*L_2)$, and so in this section we recall some facts about secant varieties and set the notation for the next section. Useful references for the following are \cite{Bertram92}, \cite{LangeNarasimhan83}, \cite{Lange84}, \cite{Schwarzenberger64}.

Consider a critical point $[L_1 \oplus L_2, \phi] \in C_d$. Then $L_2^* L_1 \otimes K$ is very ample, and so there is an embedding 
\begin{equation*}
F : X \hookrightarrow \mathbb{P}H^1(L_1^* L_2) \cong \mathbb{P}H^0(L_2^* L_1 \otimes K)^* .
\end{equation*}

Any effective divisor $D = \sum_{j=1}^k m_j p_j$ of degree $N$ defines a Hecke modification
\begin{equation*}
L_1' := L_1 \left[ -D \right] ,
\end{equation*}
together with an induced homomorphism
\begin{equation*}
i^* : H^1(L_1^* L_2) \rightarrow H^1((L_1')^*L_2) .
\end{equation*}

\begin{definition}[Secant plane of total multiplicity $N$]
Given any integer $N < \deg L_1 - \deg L_2$ and an effective divisor $D = \sum_{j=1}^k m_j p_j$ of degree $N$, the \emph{secant plane of $D$ in $\mathbb{P}H^1(L_1^* L_2)$} is the plane determined by the projectivisation of
\begin{equation*}
\ker \left( H^1(L_1^* L_2) \stackrel{i^*}{\longrightarrow} H^1((L_1[-D])^*L_2) \right) .
\end{equation*}
The \emph{total multiplicity} is $N$.
\end{definition}

If $N > 2$ then it is \emph{a priori} possible that the dimension of the secant plane will be lower than expected; for example, if three points in $X$ lie on a projective line in $\mathbb{P}H^1(L_1^* L_2)$. The next lemma gives a bound on $N$ for which the secant plane is guaranteed to have the expected dimension. In particular, this result applies to all flow lines between nonminimal critical sets (see Lemma \ref{lem:flow-lines-secant}).

\begin{lemma}\label{lem:nondegenerate-secant-bound}
If $N < d_1 - d_2$, then every secant plane in $\mathbb{P}H^1(L_1^* L_2)$ corresponding to a divisor $D = \sum_k m_k p_k$ of degree $N$ is isomorphic to $\mathbb{P}^{N-1}$, and corresponds to the unique linear $\mathbb{P}^{N-1} \subset \mathbb{P}H^1(L_1^* L_2)$ that osculates to order $m_j-1$ at each $p_j \in X \subset \mathbb{P}H^1(L_1^*L_2)$.
\end{lemma}

\begin{proof}
The bound on $N$ shows that $\deg (L_1')^* L_2 = \deg L_1^* L_2 + N < 0$, and so $h^0((L_1')^*L_2) = 0$, in which case the long exact sequence for $0 \rightarrow L_1' \hookrightarrow L_1 \rightarrow \bigoplus_{j=1}^k \mathbb{C}^{m_j}_{p_j} \rightarrow 0$ reduces to
\begin{equation*}
0 \longrightarrow \mathbb{C}^N \longrightarrow H^1(L_1^* L_2) \stackrel{i^*}{\longrightarrow} H^1((L_1')^*L_2) \longrightarrow 0 .
\end{equation*}
Therefore $\dim_\mathbb{C} \ker i^* = N$, as required. 

The statement that the plane osculates to the correct order at each point $p_j \in X$ follows from \cite[Prop. 2.4]{Hitching13}.
\end{proof}

The notation for the above construction is written in terms of the notation for critical points of $f : \mathcal{M}_{Higgs}^{ss}(E) \rightarrow \mathbb{R}$ in order to be compatible with the results of the next section. To simplify the following, from now on let $L := L_1^* L_2$ with $\deg L < 0$.

If $0 < N < -\deg L$ satisfies the bound of the previous lemma, then the above construction defines an injective map
\begin{equation}\label{eqn:secant-divisor-plane}
\sec_N^L : S^N X \rightarrow \Gr(N, H^1(L)) 
\end{equation}
that takes an effective divisor of degree $N$ to $\ker i^* \subset H^1(L)$ (or equivalently, the secant plane in $\mathbb{P}H^1(L)$). 

Consider the tautological bundle $T \rightarrow \Gr(N, H^1(L))$, and let $T_0$ denote the complement of the zero section in $T$. There is a projection $\pi : T_0 \rightarrow \mathbb{P} H^1(L)$ which takes each point in a plane to its image in $\mathbb{P} H^1(L)$, and therefore there is a projection map
\begin{equation*}
(\sec_N^L)^* T_0 \rightarrow \mathbb{P}H^1(L) .
\end{equation*}

\begin{definition}[$N^{th}$ secant variety]\label{def:secant-variety}
Let $N$ satisfy the bound of Lemma \ref{lem:nondegenerate-secant-bound}. The \emph{$N^{th}$ secant variety of $X$ in $\mathbb{P}H^1(L)$}, denoted $\Sec_N^L(X)$, is the image of the projection $(\sec_N^L)^* T_0 \rightarrow \mathbb{P}H^1(L)$.

For each effective divisor $D = \sum_{j=1}^k m_j p_j$ of degree $N$, the secant plane of $D$ in $\mathbb{P}H^1(L)$ is denoted $\Pi_D^L \subset \Sec_N^L(X) \subset \mathbb{P}H^1(L)$. The associated plane in $H^1(L)$ is denoted $\widetilde{\Pi_D^L}$.
\end{definition}

\begin{remark}
The relationship between this construction and Schwarzenberger's secant bundle construction \cite{Schwarzenberger64} is explained in \cite{Coppens04}.

The construction above shows that $\Sec_N^L(X) \subset \mathbb{P}H^1(L)$ is the subvariety containing all of the points that lie in secant planes of total multiplicity $N$.
\end{remark}

Lemma \ref{lem:nondegenerate-secant-bound} then implies the following result.

\begin{corollary}\label{cor:linearly-independent-secants}
If $\deg D < d_1 - d_2 - 1$ and $p \in X$ is not in the support of $D$, then $\Pi_p^L \cap \Pi_D^L = \{ 0 \}$.
\end{corollary}

\begin{proof}
Let $N = \deg D$. Since $D + p$ satisfies the bound of Lemma \ref{lem:nondegenerate-secant-bound}, then $\Pi_{D+p}^L \cong \mathbb{P}^N$ and $\Pi_D^L \cong \mathbb{P}^{N-1}$. Since $p$ is not in the support of $D$, then $\Pi_{D+p}^L$ is the plane spanned by $\Pi_D^L$ and $p$. Therefore $p \notin \Pi_D^L$.
\end{proof}

\begin{lemma}\label{lem:intersection-secant-planes}
Let $L \rightarrow X$ be a line bundle with $\deg L < 0$ and let $D_1$ and $D_2$ be effective divisors on $X$ with $\deg D_1 + \deg D_2 < - \deg L$. Then the secant planes defined by $D_1$ and $D_2$ will intersect in $\mathbb{P}H^1(L)$ if and only if there exists an effective divisor $E$ such that $D_1 - E \geq 0$, $D_2 - E \geq 0$. The intersection $\widetilde{\Pi_{D_1}^{L}} \cap \widetilde{\Pi_{D_2}^{L}}$ is the plane $\widetilde{\Pi_E^{L}}$ for the maximal such choice of $E$.
\end{lemma}

\begin{proof}
Since $\deg D_1 + \deg D_2 < - \deg L$, then inductively applying Corollary \ref{cor:linearly-independent-secants} shows that if $D_1$ and $D_2$ have disjoint support, then $\widetilde{\Pi_{D_1}^{L}} \cap \widetilde{\Pi_{D_2}^{L}} = \{0\}$. Therefore, if the intersection of the two planes $\widetilde{\Pi_{D_1}^{L}}$ and $\widetilde{\Pi_{D_2}^{L}}$ has positive dimension, then there is an effective divisor $E > 0$ such that $E < D_1$ and $E < D_2$ such that  
\begin{equation}\label{eqn:subspace-of-intersection}
\widetilde{\Pi_E^{L}} \subset \widetilde{\Pi_{D_1}^{L}} \cap \widetilde{\Pi_{D_2}^{L}} ,
\end{equation}
and the above inclusion is an equality if $E$ is the maximal divisor such that $E < D_1$ and $E < D_2$.

Conversely, if there exists an effective divisor $E > 0$ with $E < D_1$ and $E < D_2$, then $\widetilde{\Pi_E^{L}} \subset \widetilde{\Pi_{D_1}^{L}}$ and $\widetilde{\Pi_E^{L}} \subset \widetilde{\Pi_{D_2}^{L}}$, so \eqref{eqn:subspace-of-intersection} is satisfied.
\end{proof}

The results of the next section show that spaces of flow lines can be parametrised by secant varieties in $\mathbb{P}H^1(L_1^*L_2)$ (cf. Lemma \ref{lem:flow-lines-secant}). The following open subset of points that do not lie on any secant planes of smaller dimension will parametrise unbroken flow lines
\begin{equation}\label{eqn:nondegenerate-secants}
\Sec_{N,0}^L(X) := \Sec_N^L(X) \setminus \Sec_{N-1}^L(X) .
\end{equation} 

Numerous authors have studied the singularities in $\Sec_N^L(X)$; see for example \cite{Bertram92}, \cite{Lange84}, \cite{Schwarzenberger64}. The precise statement we need in the sequel is the main theorem of \cite{Coppens04}.

\begin{lemma}[\cite{Coppens04}]\label{lem:secant-dimension}
If $N$ satisfies the bound of Lemma \ref{lem:nondegenerate-secant-bound}, then $\Sing \left( \Sec_N^L(X) \right) =  \Sec_{N-1}^L(X)$. In particular, $\Sec_{N,0}^L(X)$ is a smooth manifold of complex dimension $\dim_\mathbb{C} \Sec_{N,0}^L(X) = 2N-1$. 
\end{lemma}

The above construction will be used to parametrise spaces of flow lines that appear within the unstable manifold of a single critical point $[L_1 \oplus L_2, \phi]$. This determines a fibre bundle over each critical set, with each fibre given by a secant variety as above for the bundle $L_1^*L_2$. This is made precise in the following definition.

\begin{definition}[Global secant variety]\label{def:secant-bundle}
Let $C_u$ be a nonminimal critical set and let $\mathbb{P}W_u^-$ be the projectivisation of the unstable manifold from Lemma \ref{lem:classify-unstable-set}. The \emph{$N^{th}$ global  secant variety over $C_u$} is the smooth fibre bundle $\mathcal{P}_{u-N}^u \hookrightarrow \mathbb{P}W_u \rightarrow C_u$ for which the fibre over $[L_1 \oplus L_2, \phi] \in C_u$ is $\Sec_{N}^{L_1^*L_2}(X) \subset \mathbb{P} H^1(L_1^*L_2)$.

The open subset for which the fibres are $\Sec_{N,0}^{L_1^*L_2}(X)$ is denoted $\mathcal{P}_{u-N,0}^u \subset \mathcal{P}_{u-N}^u$.
\end{definition}

\begin{remark}
The difference $\mathcal{P}_{u-N}^u \setminus \mathcal{P}_{u-N,0}^u$ is a fibre bundle over $C_u$ with fibre $\Sec_{N-1}^{L_1^*L_2}(X)$.
\end{remark}

The points in $\mathcal{P}_{u-N}^u$ parametrise a subset of the $\mathbb{C}^*$ equivalence classes in $W_{u,0}^-$. In order to parametrise a subset of the associated sphere bundle $S_u^- \subset \mathcal{M}_{Higgs}^{ss}(E)$, define the fibre bundles $\mathcal{S}_{u-N}^u \rightarrow C_u$ and $\mathcal{S}_{u-N,0}^u \rightarrow C_u$ by pullback
\begin{equation*}
\begin{tikzcd}
\mathcal{S}_{u-N,0}^u \arrow[hookrightarrow]{r} \arrow{d}{/S^1} & \mathcal{S}_{u-N}^u \arrow[hookrightarrow]{r} \arrow{d}{/S^1} & S_u^- \arrow{d}{/S^1} \\
\mathcal{P}_{u-N,0}^u \arrow[hookrightarrow]{r} & \mathcal{P}_{u-N}^u \arrow[hookrightarrow]{r} & \mathbb{P} W_{u,0} .
\end{tikzcd}
\end{equation*}

Given a fixed effective divisor $D$ of degree $N$, the bundles $\mathcal{P}_{u-N}^u$ and $\mathcal{S}_{u-N}^u$ each contain subbundles corresponding to the secant planes associated to $D$. In the sequel these will be denoted $\mathcal{P}_{D}^u$ and $\mathcal{S}_{D}^u$ respectively.

\begin{remark}
Let $\ell = u-N$. The results of the next section will show that $\mathcal{S}_{\ell,0}^u$ consists of Higgs pairs that flow down to the critical set $C_\ell$ and $\mathcal{S}_{\ell}^u \setminus \mathcal{S}_{\ell,0}^u$ consists of Higgs pairs that flow down to an intermediate critical set $C_m$ for $\ell < m < u$.
\end{remark}

\section{Classification of flow lines}\label{sec:flow-line-classification}

The goal of this section is to prove Theorem \ref{thm:space-unbroken-flow-lines}, which gives a new criterion to predict the downwards limit of the flow with initial condition in the unstable set of a critical point $[L_1 \oplus L_2, \phi]$. The criterion is geometric in nature, in that it is given by the secant varieties of the embedding $X \hookrightarrow \mathbb{P}H^1(L_1^* L_2)$. This method has previously been used to classify Yang-Mills flow lines \cite{Wilkin20}, and here we show that a similar idea can be used to classify all of the flow lines in the rank two moduli space $\mathcal{M}_{Higgs}^{ss}(E)$.

\subsection{Harder-Narasimhan types in the unstable manifold}

Let $L_1, L_2$ be line bundles with $\deg L_1 > \deg L_2$ and $h^0(L_1^*L_2 \otimes M) \neq 0$, so that there exists a variation of Hodge structure $[L_1 \oplus L_2, \phi \in H^0(L_1^* L_2 \otimes M)]$ corresponding to a fixed point of the $\mathbb{C}^*$ action. Lemma \ref{lem:classify-unstable-set} shows that points in the unstable manifold for the flow are determined by extensions
\begin{equation}\label{eqn:extension-unstable}
0 \rightarrow L_2 \rightarrow E \rightarrow L_1 \rightarrow 0 .
\end{equation}
The results of \cite{Hitchin87} show that the limit of the downwards flow is determined by the Harder-Narasimhan type of $E$, and this is made more precise by Hausel and Hitchin in \cite[Sec. 4.2.3]{HauselHitchin22}. The goal of this section is to explain how the limit is determined by the geometry of the extension class $[e] \in \mathbb{P} H^1(L_1^*L_2)$ for \eqref{eqn:extension-unstable}, which leads to a description of the space of unbroken flow lines connecting two critical points in terms of secant varieties (cf. Lemma \ref{lem:flow-lines-secant}).

First note that since $\dim_\mathbb{C} X = 1$, then with respect to the extension \eqref{eqn:extension-unstable}, if $L_1' \hookrightarrow E$ is a subbundle with $\deg L_1' > \deg L_2$, the composition $L_1' \hookrightarrow E \rightarrow L_1$ makes $L_1'$ a  locally free subsheaf of $L_1$. Since both are locally free of rank one, then there is an exact sequence of sheaves
\begin{equation*}
0 \rightarrow L_1' \stackrel{i}{\hookrightarrow} L_1 \rightarrow \bigoplus_{k=1}^\ell \mathbb{C}_{p_k}^{m_k} \rightarrow 0 .
\end{equation*}
Conversely, a subsheaf $L_1' \stackrel{i}{\hookrightarrow} L_1$ lifts to a subsheaf of $E$ if $e \in H^1(L_1^* L_2)$ is in the kernel of the pullback homomorphism
\begin{equation}\label{eqn:pullback-kernel}
H^1(L_1^*L_2) \stackrel{i^*}{\longrightarrow} H^1((L_1')^*L_2) 
\end{equation}
(see \cite{NarasimhanRamanan69}). Note that since $e \neq 0$, then this descends to a condition on the equivalence class $[e] \in \mathbb{P}H^1(L_1^* L_2)$. 

The resulting subsheaf $L_1' \hookrightarrow E$ will be a subbundle if and only if it is saturated, so that there is no intermediate sheaf $L_1''$ with $\deg L_1'' > \deg L_1'$ that satisfies $L_1' \hookrightarrow L_1'' \hookrightarrow L_1$ such that $e \in H^1(L_1^*L_2)$ pulls back to zero in $H^1((L_1'')^*L_2)$. 

Since $\deg L_1 > \deg L_2$, then $L_2^*L_1 \otimes K$ is very ample and so there is an embedding of the underlying Riemann surface $F : X \hookrightarrow \mathbb{P}H^0(L_2^*L_1 \otimes K)^* \cong \mathbb{P}H^1(L_1^*L_2)$. Any point $p \in X$ determines a Hecke modification $L_1[-p] \hookrightarrow L_1$ and the image $F(p) \in \mathbb{P}H^1(L_1^*L_2)$ determines a line through the origin in $H^1(L_1^*L_2)$ which is the kernel of the pullback homomorphism $H^1(L_1^*L_2) \rightarrow H^1(L_1[-p]^*L_2)$. 

More generally, an effective divisor $D = \sum_{k=1}^n m_k p_k$ with degree $m_1 + \cdots + m_n \leq \frac{1}{2} (\deg L_1 - \deg L_2)$ determines a Hecke modification $L_1' = L_1[-D]$ and a plane in $\mathbb{P}H^1(L_1^*L_2)$ via Lemma \ref{lem:nondegenerate-secant-bound}. Any point $[e]$ in this plane determines a line in $H^1(L_1^* L_2)$ that pulls back to zero in $H^1((L_1')^*L_2)$. Equivalently, by viewing $[e] \in \mathbb{P}H^1(L_1^*L_2)$ as an extension class, then $[e]$ determines a bundle $E$ with a subsheaf $L_1'$. 

This process is summarised in the diagram below.
\begin{equation}\label{eqn:subbundle-criterion}
\begin{tikzcd}
 & & & 0 \arrow{d} & 0 \arrow{dl} \\
 & 0 \arrow{d} & & L_1' \arrow{d} \arrow{dl} \\
0 \arrow{r} & L_2 \arrow{r} \arrow{d} & E \arrow{r} \arrow{dl} & L_1 \arrow{r} \arrow{d} & 0 \\
 & L_2' \arrow{d} \arrow{dl} & &  \bigoplus_{k=1}^n \mathbb{C}_{p_k}^{m_k} \arrow{d} \\
0 &  \bigoplus_{k=1}^n \mathbb{C}_{p_k}^{m_k} \arrow{d} & & 0 \\
 & 0
\end{tikzcd}
\end{equation}

\begin{lemma}\label{lem:subbundle-criterion}
This subsheaf $L_1' \hookrightarrow E$ is a subbundle of $E$ if and only if both of the following are true
\begin{enumerate}

\item \label{item:full-subset} $[e]$ does not lie on a plane through a proper subset of the points $F(p_1), \ldots, F(p_\ell)$, and

\item \label{item:maximal-osculation} $[e]$ does not lie on a plane through $F(p_1), \ldots, F(p_\ell)$ that osculates to order strictly less than $m_j-1$ at $F(p_j)$ for at least one $j$. 
\end{enumerate}
\end{lemma}

\begin{proof}
The subsheaf $L_1' \hookrightarrow E$ is a subbundle if and only if it is saturated. This occurs if and only if there is no effective divisor $D' < D$ such that $[e]$ lies on the plane $\Pi_{D'}^{L_1^*L_2}$. The condition $D' < D$ shows that the plane $\Pi_{D'}$ is either a plane through a proper subset of the points $F(p_1), \ldots, F(p_\ell)$, or a plane through $F(p_1), \ldots, F(p_\ell)$ that osculates to order strictly less than $m_j-1$ at $F(p_j)$ for at least one $j$. Therefore $L_1' \hookrightarrow E$ is saturated if and only if conditions \eqref{item:full-subset} and \eqref{item:maximal-osculation} are both true.
\end{proof}

In particular, if $\deg D = \sum_k m_k < \frac{1}{2} (\deg L_1 - \deg L_2)$ such that $\deg L_1' > \frac{1}{2} \deg E$, then the uniqueness of the Harder-Narasimhan filtration shows that there is a unique such subbundle of maximal degree.

\begin{lemma}\label{lem:HN-type-extension}
Let $0 \rightarrow L_2 \rightarrow E \rightarrow L_1 \rightarrow 0$ be an extension with $\deg L_1 > \deg L_2$, let $D$ be an effective divisor such that $\deg D < \frac{1}{2} (\deg L_1 - \deg L_2)$ and suppose that the extension class lies in the kernel of the pullback \eqref{eqn:pullback-kernel}. If $[e] \in \mathbb{P}H^1(L_1^*L_2)$ satisfies the conditions of Lemma \ref{lem:subbundle-criterion}, then $E$ has Harder-Narasimhan filtration
\begin{equation}\label{eqn:HN-filtration}
\{0\} \subset L_1' \subset E .
\end{equation}
\end{lemma}

\begin{proof}
The diagram \eqref{eqn:subbundle-criterion} shows that $L_1'$ is a subsheaf of $E$, and conditions \eqref{item:full-subset} and \eqref{item:maximal-osculation} show that this subsheaf is a subbundle. The bound on $\deg D$ implies that $\deg L_1' > \frac{1}{2} \deg E$ and so $L_1'$ is a destabilising subbundle. 

To see that $L_1'$ is the \emph{maximal} destabilising subbundle, note that if there exists another effective divisor $D'$ with $\deg D' < \deg D$ and $L_1[-D'] \subset E$, then 
\begin{equation*}
e \in \left( \ker (H^1(L_1^*L_2) \rightarrow H^1(L_1[-D']^*L_2)) \right) \cap \left( \ker (H^1(L_1^*L_2) \rightarrow H^1(L_1[-D]^*L_2)) \right)
\end{equation*}
 implies that the two planes $\Pi_{D'}^{L_1^*L_2}$ and $\Pi_D^{L_1^*L_2}$ have nontrivial intersection. Therefore Lemma \ref{lem:intersection-secant-planes} shows that there is an effective divisor $D'' > 0$ such that $D'' < D$ and $D'' < D'$ and $e \in \ker (H^1(L_1^*L_2) \rightarrow H^1(L_1[-D'']^*L_2))$, which contradicts Lemma \ref{lem:subbundle-criterion}.

Therefore we conclude that there is no effective divisor $D'$ such that $\deg D' < \deg D$ and $e \in \ker (H^1(L_1^*L_2) \rightarrow H^1(L_1[-D']^*L_2))$, and so $L_1' = L_1[-D]$ has the largest possible degree for a subbundle of $E$. Therefore $L_1' \hookrightarrow E$ is the maximal destabilising subbundle, from which we can conclude that the Harder-Narasimhan filtration is \eqref{eqn:HN-filtration}.
\end{proof}

The following lemma relates this to the secant varieties of the previous section. 

\begin{lemma}\label{lem:flow-lines-secant}
Let $E \rightarrow X$ be a rank $2$ complex vector bundle, let $u$ be an integer in the range $\frac{1}{2} \deg E < u \leq \frac{1}{2} \left( \deg E + \deg M \right)$, let $[L_1 \oplus L_2, \phi] \in C_u$ and let $D$ be an effective divisor with $\deg D < \frac{1}{2}(\deg L_1 - \deg L_2)$. Then the subset of the unstable manifold $W_{[L_1 \oplus L_2, \phi]}^-$ consisting of pairs $[E, \phi']$ where $E$ is isomorphic to an extension of line bundles
\begin{equation*}
0 \rightarrow L_1[-D] \rightarrow E \rightarrow L_2[D] \rightarrow 0
\end{equation*}
is given by $\pi^{-1} (\Sec_{\deg D,0}^{L_1^*L_2}(X))$, where $\pi$ denotes the projection $\pi : W_{[L_1 \oplus L_2, \phi]}^- \setminus \{0\} \rightarrow \mathbb{P}W_{[L_1 \oplus L_2, \phi]}^-$.
\end{lemma}

\subsection{Higgs pairs in the limit of the downwards flow}\label{subsec:explicit-gauge}

The results of the previous section classify the Harder-Narasimhan filtrations in each unstable set $W_u^-$, however this only determines the holomorphic bundle underlying the Higgs pair in the lower limit of the flow, which is determined by the extension from diagram \eqref{eqn:subbundle-criterion}. The goal of this section is to give an explicit description of the gauge transformation that determines this isomorphism. This is motivated by the constructions in \cite[Sec. 4.5]{Witten18} and \cite{Wilkin20} which also give an explicit description of the effect of a Hecke modification on a Higgs field. The difference here is that one can compose the two Hecke modifications that appear in the diagram \eqref{eqn:subbundle-criterion} in such a way as to construct a smooth gauge transformation which induces an isomorphism of Higgs pairs. 

On writing the Higgs pair as an extension of bundles $0 \rightarrow L_2 \rightarrow E \rightarrow L_1 \rightarrow 0$ with Higgs field $\phi \in H^0(L_1^*L_2 \otimes M)$, in this gauge it is easy to see the limit of the upwards flow, however the limit of the downwards flow is not obvious. After changing gauge in this way, from this new point of view the limit of the downwards flow appears via a simple calculation (see Lemma \ref{lem:downwards-flow-criterion} below).

In preparation for Lemma \ref{lem:downwards-flow-criterion}, first consider the case of a Hecke modification of multiplicity $m$ at a single point $p \in X$. Let $U$ be a coordinate neighbourhood with coordinate $z$ centred at $p$. Choose open neighbourhoods $V, W$ of $p$ with $\overline{W} \subset V$, $\overline{V} \subset U$ and a $C^\infty$ bump function $\eta : X \rightarrow \mathbb{R}_{\geq 0}$ such that $\eta \equiv 0$ on $W$ and $\eta \equiv 1$ on $X \setminus V$. Then
\begin{equation*}
\omega = \frac{\bar{\partial} \eta}{z^m} \in \Omega^{0,1}(X)
\end{equation*}
defines a smooth one form on $X$. In fact, since the support of $\omega$ is contained in a coordinate neighbourhood, then $\omega \in  \Omega^{0,1}(L)$ for any line bundle $L$ trivialised over $U$; in particular $[\omega]$ defines a cohomology class in $H^{0,1}(L_1^*L_2)$.

Let $\gamma : [0,1] \rightarrow X$ be a loop contained in $U \setminus V$ which has winding number one around the point $p$. Serre duality then identifies the Dolbeault cohomology class $[\omega] \in H^{0,1}(L_1^*L_2)$ with an element of $H^0(L_2^*L_1 \otimes K)^*$ given by
\begin{multline}\label{eqn:residue-class}
[\omega](s) := \int \hspace{-0.12 in} \int_X \frac{(\bar{\partial} \eta) s}{z^m} = \int \hspace{-0.12 in} \int_U \frac{(\bar{\partial} \eta) s}{z^m} \\ 
= \int_\gamma \frac{\eta s}{z^m} = \int_\gamma \frac{s}{z^m} = 2 \pi i \Res_{z=0} \left( \frac{s}{z^m} \right) \quad \text{for all $s \in H^0(L_2^*L_1 \otimes K)$} .
\end{multline}

% Note that the application of Stokes theorem in the above construction requires $\frac{\eta s}{z^m}$ to extend smoothly across the origin, otherwise there is an extra term.

When $m=1$, then this is a scalar multiple of the evaluation map at the point $p$, and therefore $[\omega] \in H^{0,1}(L_1^*L_2) \cong H^0(L_2^*L_1 \otimes K)^*$ corresponds to the image of $p\in X$ in $\mathbb{P}H^0(L_2^*L_1 \otimes K)^*$. When $m > 1$ then \eqref{eqn:residue-class} determines an $m$-dimensional subspace 
\begin{equation*}
\spanC \{ [\omega], [z \omega], \ldots, [z^{m-1} \omega] \} \subset H^0(L_2^*L_1 \otimes K)^*
\end{equation*}
which corresponds to the plane in $\mathbb{P}H^0(L_2^*L_1 \otimes K)^*$ that osculates to order $m-1$ to the image of $X$ at the point $p \in X$.

Now consider an extension $0 \rightarrow L_2 \rightarrow E \rightarrow L_1 \rightarrow 0$ with extension class $[\omega]$. On the trivalisation over $U$, this corresponds to the holomorphic structure
\begin{equation*}
\bar{\partial}_A = \bar{\partial} + \left( \begin{matrix} 0 & \frac{\bar{\partial} \eta}{z^m} \\ 0 & 0 \end{matrix} \right) ,
\end{equation*}
where the matrix (which is only nontrivial over $U$) is defined using the basis on $\left. E \right|_U$ given by the $C^\infty$ bundle isomorphism $\left. E \right|_U \cong \left. L_2 \right|_U \oplus \left. L_1 \right|_U$. Let $\varphi \in H^0(L_1^*L_2 \otimes M) \setminus \{0\}$ so that 
\begin{equation*}
\phi = \left( \begin{matrix} 0 & \varphi \\ 0 & 0 \end{matrix} \right) .
\end{equation*}
Then Lemma \ref{lem:classify-unstable-set} shows that $[\bar{\partial}_A, \phi]$ is a Higgs pair in the unstable manifold of the critical point $[L_1 \oplus L_2, \phi]$.

Let $L_1' = L_1[-mp]$. On a trivialisation over the neighbourhood $U$, the pullback of $\frac{\bar{\partial} \eta}{z^m}$ to $\Omega^{0,1}((L_1')^*L_2)$ is given by applying a meromorphic gauge transformation
\begin{equation*}
\left( \begin{matrix} 1 & 0 \\ 0 & z^{-m} \end{matrix} \right) \left( \begin{matrix} 0 & \frac{\bar{\partial} \eta}{z^m} \\ 0 & 0 \end{matrix} \right) \left( \begin{matrix} 1 & 0 \\ 0 & z^m \end{matrix} \right) = \left( \begin{matrix} 0 & \bar{\partial} \eta \\ 0 & 0 \end{matrix} \right)
\end{equation*}
to obtain an exact $(0,1)$ form. To trivialise this, one can apply a smooth gauge transformation
\begin{equation*}
\left( \begin{matrix} 1 & \eta-1 \\ 0 & 1 \end{matrix} \right) \left( \begin{matrix} 0 & \bar{\partial} \eta \\ 0 & 0 \end{matrix} \right) \left( \begin{matrix} 1 & 1-\eta \\ 0 & 1 \end{matrix} \right) - \left( \begin{matrix} 0 & \bar{\partial} \eta \\ 0 & 0 \end{matrix} \right) \left( \begin{matrix} 1 & 1-\eta \\ 0 & 1 \end{matrix} \right) = \left( \begin{matrix} 0 & 0 \\ 0 & 0 \end{matrix} \right) 
\end{equation*}
to obtain a trivial holomorphic structure over $U$. Therefore, gauging by the singular gauge transformation
\begin{equation*}
g_1 = \left( \begin{matrix} 1 & \eta-1 \\ 0 & 1 \end{matrix} \right) \left( \begin{matrix} 1 & 0 \\ 0 & z^{-m} \end{matrix} \right) = \left( \begin{matrix} 1 & z^{-m} (\eta-1) \\ 0 & z^{-m} \end{matrix} \right)
\end{equation*}
on a trivialisation over $U$ shows that the pullback of the extension $0 \rightarrow L_2 \rightarrow E \rightarrow L_1 \rightarrow 0$ is a direct sum $L_2 \oplus L_1'$.

In a similar way, on the trivialisation over $U$ one can gauge $L_2 \oplus L_1'$ by another singular gauge transformation
\begin{equation*}
g_2 = \left( \begin{matrix} z^m & 0 \\ 0 & 1 \end{matrix} \right) \left( \begin{matrix} 1 & 0 \\ 1-\eta & 1 \end{matrix} \right) = \left( \begin{matrix} z^m & 0 \\ 1-\eta & 1 \end{matrix} \right)
\end{equation*}
to obtain an extension $0 \rightarrow L_1[-mp] \rightarrow E' \rightarrow L_2[mp] \rightarrow 0$ corresponding to the diagonal exact sequence in the diagram \eqref{eqn:subbundle-criterion}. \emph{A priori} this bundle $E'$ may not be isomorphic to $E$, however the composition of these gauge transformations is
\begin{equation}\label{eqn:smooth-composition}
g_2 g_1 = \left( \begin{matrix} z^m & \eta-1 \\ 1-\eta & z^{-m} (1-(1-\eta)^2) \end{matrix} \right) ,
\end{equation}
which is a $C^\infty$ gauge transformation, since the point $\{z=0\}$ is contained in the neighbourhood $W$ where $\eta \equiv 0$, in which case we have
\begin{equation*}
\left. g_2 g_1 \right|_{W} = \left( \begin{matrix} z^m & 1 \\ -1 & 0 \end{matrix} \right) .
\end{equation*}
In particular, $g_2 g_1$ defines an isomorphism of Higgs pairs $(E', \phi') := g_2 g_1 \cdot (E, \phi)$

In summary, the above construction gives an explicit gauge theoretic construction of a family of Higgs fields in $\mathcal{B}^{ss}$ that represent a flow line in $\mathcal{M}_{Higgs}^{ss}(E)$, from which we have a new proof of the results of \cite[Sec. 4.2.3]{HauselHitchin22} on the limit of the downwards flow.

\begin{lemma}\label{lem:downwards-flow-criterion}
Let $[L_1 \oplus L_2, \phi_\infty] \in C_u$ and consider a Higgs pair in the unstable manifold $W_u^-$ for which the underlying holomorphic bundle is an extension $0 \rightarrow L_2 \rightarrow E \rightarrow L_1 \rightarrow 0$ which is isomorphic to $0 \rightarrow L_1[-D] \rightarrow E \rightarrow L_2[D] \rightarrow 0$ via \eqref{eqn:subbundle-criterion} for an effective divisor $D$ such that $\deg D < \frac{1}{2}(\deg L_1 - \deg L_2)$. Then the limit of the downwards flow is the critical point $[L_1[-D] \oplus L_2[D], \phi_0]$, where $\phi_0 \in H^0(L_1^* L_2[2D] \otimes M)$ is the image of $\phi_\infty \in H^0(L_1^*L_2 \otimes M)$ via the sheaf homomorphism $L_1^*L_2 \otimes M \hookrightarrow L_1^* L_2[2D] \otimes M$.
\end{lemma}

\begin{proof}
In a local trivialisation around each point $p_k$ with multiplicity $m_k$, the construction above determines a smooth gauge transformation $g_2 g_1$ from \eqref{eqn:smooth-composition} such that
\begin{equation*}
\phi' = (g_2 g_1) \cdot \phi \in H^0((L_1[m_kp_k])^* L_2[m_kp_k] \otimes K) .
\end{equation*}
A computation shows that
\begin{align*}
(g_2 g_1) \cdot \phi & = \left( \begin{matrix} z^{m_k} & \eta-1 \\ 1-\eta & z^{-m_k} (1 - (1-\eta)^2) \end{matrix} \right) \left( \begin{matrix} 0 & \varphi \\ 0 & 0 \end{matrix} \right) \left( \begin{matrix} z^{-m_k} (1 - (1-\eta)^2) & 1-\eta \\ \eta-1 & z^{m_k} \end{matrix} \right) \\
 & = \left( \begin{matrix} (\eta-1) z^{m_k} \varphi & z^{2m_k} \varphi \\ -(1-\eta)^2 \varphi & (1-\eta) z^{m_k} \varphi \end{matrix} \right)
\end{align*}
Now let $[L_1' \oplus L_2', \phi_0]$ denote the critical point at the lower limit of the downwards flow. This flow is given by scaling the Higgs field $\phi \mapsto e^{-2t} \phi$ and applying a complex gauge transformation to preserve Hitchin's equations \eqref{eqn:hitchin-equations}. In the situation under consideration, this gauge transformation has the effect of scaling the extension class to zero. On the Higgs field $\phi'$ this has the form
\begin{align*}
\phi'(t) & = e^{-2t} \left( \begin{matrix} e^t & 0 \\ 0 & e^{-t} \end{matrix} \right) \left( \begin{matrix} (\eta-1) z^{m_k} \varphi & z^{2m_k} \varphi \\ -(1-\eta)^2 \varphi & (1-\eta) z^{m_k} \varphi \end{matrix} \right) \left( \begin{matrix} e^{-t} & 0 \\ 0 & e^t \end{matrix} \right) \\
 & = \left( \begin{matrix} e^{-2t} (\eta-1) z^{m_k} \varphi &  z^{2m_k} \varphi \\ -e^{-4t} (1-\eta)^2 \varphi & e^{-2t} (1-\eta) z^{m_k} \varphi \end{matrix} \right) \\
\Rightarrow \quad \lim_{t \rightarrow \infty} \phi'(t) & = \left( \begin{matrix} 0 &  z^{2m_k} \varphi \\ 0 & 0 \end{matrix} \right) =: \phi_0 .
\end{align*}
Therefore we see that the Higgs field in the limit of the flow now has an extra zero of order $2m_k$ at each point $p_k$.
 
The general case works in the same way, by using disjoint coordinate neighbourhoods $U_k$ of each $p_k$ to construct an extension class associated to each effective divisor $D = \sum_{k=1}^n m_k p_k$ satisfying $\deg D < \frac{1}{2} (\deg L_1 - \deg L_2)$. The same process as above shows that the limit $\phi_0$ must have a zero of order $2 m_k$ at each point $p_k \in X$ for all $k = 1, \ldots, n$. 
\end{proof}

Conversely, given a critical point $[L_1' \oplus L_2', \phi_0] \in C_\ell$, the same method can be used to describe all points in the upwards limit of a flow line emanating from $[L_1' \oplus L_2', \phi_0]$. This result was obtained using a different method of proof by Hausel and Hitchin in \cite[Sec. 4.2.3]{HauselHitchin22}.

\begin{corollary}\label{cor:upwards-flow-criterion}
Let $[L_1' \oplus L_2', \phi_0] \in C_\ell$, let $D$ be an effective divisor with degree bounded by 
\begin{equation*}
0 < \deg D < \frac{1}{2} \left( \deg E + \deg M \right) - \ell
\end{equation*}
and suppose that $\phi_0 \in H^0((L_1')^*L_2' \otimes M)$ is the image of some $\phi_0 \in H^0(L_1'^* L_2'[-2D] \otimes M)$ under the homomorphism $H^0((L_1')^* L_2'[-2D] \otimes M) \hookrightarrow H^0((L_1')^*L_2' \otimes M)$. Then there exists a flow line connecting $[L_1' \oplus L_2', \phi_0] \in C_\ell$ and $[L_1'[D] \oplus L_2'[-D], \phi_0] \in C_{\ell+\deg D}$.
\end{corollary}

\begin{proof}
Lemma \ref{lem:downwards-flow-criterion} shows that for all effective divisors $D$ satisfying the above degree bound, one can construct a flow line between $[L_1'[D] \oplus L_2'[-D], \phi_\infty]$ and $[L_1' \oplus L_2', \phi_0] \in C_\ell$.
\end{proof} 

\subsection{Classification of flow lines}

Using the results of the previous section, we can now classify all the flow lines between two critical points.
\begin{corollary}\label{cor:pointwise-flow-lines}
Let $[L_1' \oplus L_2', \phi_0] \in C_d$ and $[L_1'[D] \oplus L_2'[-D], \phi_\infty] \in C_{d+\deg D}$ be critical points with $\phi_0 \in H^0((L_1')^*L_2' \otimes M)$ the image of $\phi_\infty \in H^0(L_1'^* L_2'[-2D] \otimes M)$ under the homomorphism $H^0(L_1'^* L_2'[-2D] \otimes M) \hookrightarrow H^0((L_1')^*L_2' \otimes M)$. Modulo the $S^1$ action, the space of all flow lines between these critical points is parametrised by the open subset 
\begin{equation*}
\Pi_D^{L_1'^* L_2'[-2D]} \cap \Sec_0^N(X) \subset \Pi_D^{L_1'^* L_2'[-2D]} \subset \mathbb{P} H^1((L_1')^* L_2'[-2D]) .
\end{equation*}
\end{corollary}

\begin{proof}
Lemma \ref{lem:HN-type-extension} shows that the secant planes in the projectivisation of the unstable manifold determine the Harder-Narasimhan type. Lemma \ref{lem:downwards-flow-criterion} and Corollary \ref{cor:upwards-flow-criterion} show that the Harder-Narasimhan type determines the limit of the downwards flow, from which we obtain the desired result.
\end{proof}

In order to state the next theorem, let $C_\ell$ and $C_u$ be two critical sets with $f(C_\ell) < f(C_u)$, and recall the sphere bundle $S_u^- \rightarrow C_u$ from \eqref{eqn:neg-sphere-bundle}.  Since $f$ is $S^1$ invariant, then the level sets and the unstable sets are preserved by the $S^1$ action, and so there is a canonical homeomorphism
\begin{equation*}
S_u^- / S^1 \cong \mathbb{P}W_u^- .
\end{equation*}
Recall from \eqref{eqn:unbroken-flow-lines} that the space of unbroken flow lines from $C_\ell$ to $C_u$ is denoted by $\mathcal{L}_\ell^u$. Then there is an inclusion $\mathcal{L}_\ell^u \hookrightarrow S_u^-$ given by mapping a flow line to the unique point of intersection with the level set $f^{-1} \left( f(C_u) - \varepsilon \right)$. 

It follows directly from Definition \ref{def:secant-bundle} that there is an inclusion of the global secant variety
\begin{equation*}
\mathcal{P}_\ell^u \hookrightarrow \mathbb{P}W_u^- .
\end{equation*}

The next result relates the space of flow lines $\mathcal{L}_\ell^u$ to the global secant variety $\mathcal{P}_\ell^u$.

\begin{theorem}\label{thm:space-unbroken-flow-lines}
The projection $S_u^- \rightarrow S_u^- / S^1 \cong \mathbb{P}W_u^-$ maps the image of $\mathcal{L}_\ell^u \hookrightarrow S_u^-$ to the image of $\mathcal{P}_\ell^u \hookrightarrow \mathbb{P}W_u^-$. In particular, this projection induces a circle bundle $g : \mathcal{L}_\ell^u \rightarrow \mathcal{P}_\ell^u$, where the fibres are orbits of the $S^1$ action $e^{i \theta} \cdot [E, \phi] = [E, e^{i\theta} \phi]$.
\end{theorem}

\begin{proof}
Corollary \ref{cor:pointwise-flow-lines} determines the space of flow lines between two critical points. Extending this to the flow lines within the unstable bundle over the critical set $C_u$ gives the desired result.
\end{proof}

Now let $\overline{\mathcal{L}_\ell^u}$ denote the closure of $\mathcal{L}_\ell^u$ in $S_u^-$. This corresponds to adding flow lines for which the downwards limit lies in an intermediate critical set $C_m$ 
\begin{equation}\label{eqn:flow-compactification1}
\overline{\mathcal{L}_\ell^u} \setminus \mathcal{L}_\ell^u = \bigcup_{\ell < m < u} ( \mathcal{L}_m^u \cap \overline{\mathcal{L}_\ell^u} ) .
\end{equation}

Similarly, let $\overline{\mathcal{P}_\ell^u}$ denote the closure of $\mathcal{P}_\ell^u$ in $\mathbb{P}W_u^-$. This corresponds to adding points lying on secant planes of lower dimension
\begin{equation}\label{eqn:secant-compactification1}
\overline{\mathcal{P}_\ell^u} \setminus \mathcal{P}_\ell^u = \bigcup_{\ell < m < u} \mathcal{P}_m^u .
\end{equation}

Applying Theorem \ref{thm:space-unbroken-flow-lines} to the spaces $\mathcal{L}_m^u$ for $\ell < m < u$ gives us the following result, which shows that the projection map $\mathcal{L}_\ell^u \rightarrow \mathcal{P}_\ell^u$ from Theorem \ref{thm:space-unbroken-flow-lines} extends to a projection on the closure.

\begin{proposition}\label{prop:flow-secant-closure}
Restricting the projection $S_u^- \rightarrow S_u^- / S^1 \cong \mathbb{P}W_u^-$ to $\overline{\mathcal{L}_\ell^u} \subset S_u^-$ determines a circle bundle $\overline{\mathcal{L}_\ell^u} \rightarrow \overline{\mathcal{P}_\ell^u}$.
\end{proposition}

\section{The energy function is Morse-Bott-Smale}\label{sec:Morse-Bott-Smale}

Hitchin \cite{Hitchin87} uses a general result of Frankel \cite{Frankel59} to show that if the degree and rank of $E$ are coprime, then $f : \mathcal{M}_{Higgs}^{st} \rightarrow \mathbb{R}$ is a perfect Morse-Bott function (see also \cite{feehan22} and \cite{feehanleness23} for related results for Bialynicki-Birula stratifications of singular moduli spaces). When $\rank(E) = 2$, then this also applies in the noncoprime case, since the singularities in the moduli space do not intersect the stable or unstable manifolds for the nonminimal critical sets (see Section \ref{subsec:rank-2-remarks}). Now we can use the explicit description of the spaces of flow lines to show that the stable and unstable manifolds of $f$ intersect transversely, and therefore $f$ satisfies the stronger Morse-Bott-Smale condition.

Let $\rank(E) = 2$, let $\frac{1}{2} \deg E < \ell < u \leq \frac{1}{2} \left( \deg E + \deg M \right)$ and consider two critical sets $C_{\ell}$, $C_{u}$. Since $f$ is Morse-Bott, then $W_{\ell}^+$ is a manifold with codimension equal to the Morse index at $C_{\ell}$. Therefore, for any $x \in \mathcal{F}_{\ell}^{u}$ we have
\begin{equation*}
T_x \mathcal{M}_{Higgs}^{ss} \cong T_x W_{\ell}^+ \oplus N_x^\mathcal{M} W_{\ell}^+ ,
\end{equation*}
where $N_x^\mathcal{M} W_{\ell}^+$ denotes the normal bundle to $T_x W_{\ell}^+$ in the ambient manifold $\mathcal{M} := \mathcal{M}_{Higgs}^{ss}$. Theorem \ref{thm:space-unbroken-flow-lines} shows that the space of unbroken flow lines between the two critical sets is $\mathcal{S}_\ell^{u} \subset S_u^-$. In particular, we have the following

\begin{lemma}\label{lem:index-equals-codimension}
If $\ell > 0$, then the real codimension of $\mathcal{F}_{\ell}^{u}$ in $W_{u}^-$ is equal to the Morse index of $C_{\ell}$.
\end{lemma}

\begin{proof}
The codimension of the global secant variety $\mathcal{P}_\ell^{u} \subset \mathbb{P} W_{u}^-$ is equal to the codimension of $\Sec_{u-\ell, 0}^{L_1^*L_2}(X)$ in a single fibre $\mathbb{P}H^1(L_1^*L_2)$.  Lemma \ref{lem:secant-dimension} shows that $\dim_\mathbb{R} \Sec_0^{u-\ell}(X) = 2 \left( 2(u-\ell) - 1 \right)$ and so an application of Riemann-Roch shows that the real codimension in $\mathbb{P}H^1(L_1^*L_2)$ is
\begin{align*}
\dim_\mathbb{R} H^1(L_1^*L_2) - 2 - 2(2(u-\ell) - 1) & = 2 \left( g-1 + \deg L_1 - \deg L_2 - 2(u-\ell) \right) \\
 & = 2 \left( g-1 + u - (\deg E - u) - 2(u-\ell) \right) \\ 
 & = 2 \left( g-1  - \deg E  + 2 \ell \right) ,
\end{align*}
which is equal to the Morse index $2(g-1 + \ell - (\deg E - \ell))$ at $C_\ell$ from \eqref{eqn:Morse-index}.
\end{proof}

\begin{proposition}\label{prop:morse-bott-smale}
When $\rank(E) = 2$, then the function $f : \mathcal{M}_{Higgs}^{ss}(E) \rightarrow \mathbb{R}$ is Morse-Bott-Smale.
\end{proposition}

\begin{proof}
If $\ell = 0$, then $W_\ell^+$ is an open dense subset of $\mathcal{M}_{Higgs}^{ss}(E)$ and so the intersection $W_\ell^+ \cap W_u^-$ is automatically transverse for all $u > 0$.

If $\ell > 0$, then the previous lemma shows that there is a subspace $N_x^{W_{u}^-} \mathcal{F}_{\ell}^{u} \subset T_x W_{u}^-$ which is complementary to $T_x W_{\ell}^+$ and has the same dimension as $N_x^\mathcal{M} W_{\ell}^+$. Therefore we have
\begin{equation*}
T_x \mathcal{M}_{Higgs}^{ss} \cong T_x W_\ell^+ \oplus N_x^\mathcal{M} W_\ell^+ \cong T_x W_{\ell}^+ \oplus N_x^{W_{u}^-} \mathcal{F}_{\ell}^{u} \subset T_x W_{\ell}^+ + T_x W_{u}^- ,
\end{equation*}
and so the intersection $W_{\ell}^+ \cap W_{u}^-$ is transverse.
\end{proof}

\section{Compactification of spaces of flow lines}\label{sec:compactification}

An important step in the construction of the Morse-Bott-Smale complex (cf. \cite{AustinBraam95}) is to compactify the space of flow lines $\mathcal{L}_\ell^u$ between two critical sets by adding broken flow lines. For the moduli space of Higgs bundles, Theorem \ref{thm:space-unbroken-flow-lines} shows that $\mathcal{L}_\ell^u$ has an algebro-geometric interpretation in terms of the global secant variety $\mathcal{P}_\ell^u$ and the goal of this section is to prove Theorem \ref{thm:Morse-resolution}, which shows that the compactification by broken flow lines also has an algebro-geometric interpretation via the resolution of secant varieties studied by Bertram \cite{Bertram92}.

Recall the space $\mathcal{L}_\ell^u$ of unbroken flow lines from \eqref{eqn:unbroken-flow-lines} and the inclusion $\mathcal{L}_\ell^u \hookrightarrow S_u^-$ into the sphere bundle \eqref{eqn:neg-sphere-bundle}. There are two compactifications of $\mathcal{L}_\ell^u$ that will be important in the sequel. The first, denoted by $\overline{\mathcal{L}_\ell^u}$, is simply given by taking the closure in $S_u^-$, and corresponds to taking the union of $\mathcal{L}_\ell^u$ with all spaces $\mathcal{L}_m^u \cap \overline{\mathcal{L}_\ell^u}$ such that $\ell < m < u$ (cf. \eqref{eqn:flow-compactification1}).

The second compactification corresponds to adding broken flow lines between $C_u$ and $C_\ell$. Austin and Braam \cite{AustinBraam95} give a detailed description of this compactification, which will be denoted $\widetilde{\mathcal{L}_\ell^u}$ in the sequel (see Section \ref{subsec:Morse-resolution} for more details). There is a canonical projection
\begin{equation}\label{eqn:Morse-resolution}
P_{Morse} : \widetilde{\mathcal{L}_\ell^u} \rightarrow \overline{\mathcal{L}_\ell^u}
\end{equation}
given by mapping a broken flow line emanating from $C_u$ to the unique point of intersection with the level set $f^{-1} \left( f(C_u) - \varepsilon \right)$, where $\varepsilon > 0$ is chosen so that there are no critical values between $f(C_u) - \varepsilon$ and $f(C_u)$. This projection is one-to-one on the open subset $\mathcal{L}_\ell^u \subset \widetilde{\mathcal{L}_\ell^u}$ and for each $x \in \widetilde{\mathcal{L}_\ell^u} \setminus \mathcal{L}_\ell^u$ such that $\lim_{t \rightarrow \infty} \varphi(x,t) = x_\infty \in C_m$ for $\ell < m < u$, the fibre $P_{Morse}^{-1}(x)$ is the space of broken flow lines between $x_\infty$ and $C_\ell$.

On the algebro-geometric side of the picture, Bertram \cite{Bertram92} constructs a resolution $\widetilde{\Sec}_L^N(X) \rightarrow \Sec_L^N(X)$ of each secant variety, which extends to a fibrewise resolution 
\begin{equation}\label{eqn:secant-resolution}
P_{Sec} : \widetilde{\mathcal{P}_\ell^u} \rightarrow \overline{\mathcal{P}_\ell^u}
\end{equation}
of the global secant variety from Definition \ref{def:secant-bundle}. In this resolution, the points in the exceptional divisors correspond to sequences of points in secant varieties of $X \subset \mathbb{P}H^1(L_1^*L_2[2D])$ for different divisors $D$ (see \cite[Cor. 2.5(b)]{Bertram92} and Lemma \ref{lem:secant-fibres} below for more details). From the point of view of Theorem \ref{thm:space-unbroken-flow-lines}, this corresponds to a sequence $y_0, \ldots, y_n$ of critical points together with points $z_1 \in \mathbb{P} W_{y_0}^-, \ldots, z_n \in \mathbb{P} W_{y_n}^-$ such that each $z_k \in \mathbb{P} W_{y_{k-1}}^-$ lies in a secant plane of $X \hookrightarrow \mathbb{P} W_{y_{k-1}}^-$ such that the preimage in the sphere bundle $S_{y_{k-1}}^-$ flows down to $y_k$.

Theorem \ref{thm:space-unbroken-flow-lines} shows that the $S^1$ action on the moduli space determines a circle bundle $\mathcal{L}_\ell^u \rightarrow \mathcal{P}_\ell^u$, and Proposition \ref{prop:flow-secant-closure} shows that this extends to a circle bundle $\overline{\mathcal{L}_\ell^u} \rightarrow \overline{\mathcal{P}_\ell^u}$. In Section \ref{subsec:Morse-resolution}, we construct a map $\widetilde{\mathcal{L}_\ell^u} \rightarrow \widetilde{\mathcal{P}_\ell^u}$ that takes a broken flow line to a sequence of secant planes, corresponding to a point in the resolution \eqref{eqn:secant-resolution}. The goal of this section is to prove the following result that this map extends to a projection from the Morse resolution \eqref{eqn:Morse-resolution} to Bertram's resolution \eqref{eqn:secant-resolution}.

\begin{theorem}\label{thm:Morse-resolution}
The following diagram commutes
\begin{equation}\label{eqn:relate-resolutions}
\begin{tikzcd}[column sep=1.5cm]
\widetilde{\mathcal{L}_\ell^u} \arrow{r}{\text{Def. \ref{def:resolution-map}}} \arrow{d}[swap]{P_{Morse}} & \widetilde{\mathcal{P}_\ell^u} \arrow{d}{P_{sec}} \\
\overline{\mathcal{L}_\ell^u} \arrow{r}{\text{Prop. \ref{prop:flow-secant-closure}}} & \overline{\mathcal{P}_\ell^u} 
\end{tikzcd}
\end{equation}
\end{theorem}

One of the  motivations of \cite{Bertram92} was to resolve the rational extension map to the moduli space of semistable bundles $\mathbb{P} := \mathbb{P}H^1(L_1^*L_2) \dashrightarrow \mathcal{M}^{ss}(E)$ to a morphism $\widetilde{\mathbb{P}} \rightarrow \mathcal{M}^{ss}(E)$. Theorem \ref{thm:Morse-resolution} gives a Morse-theoretic interpretation of this morphism by showing that the Bertram's resolution of the extension map corresponds to the map $\widetilde{\mathcal{L}_0^u} \rightarrow \mathcal{M}^{ss}(E)$ taking a broken flow line to the critical point in the lower limit.

\subsection{The Morse resolution}\label{subsec:Morse-resolution}

In the following, let $f : M \rightarrow \mathbb{R}$ be a proper Morse-Bott function, with critical sets labelled $C_d$ for $0 \leq d \leq n$ and $f(C_i) < f(C_j)$ iff $i < j$. We also assume that $f$ is \emph{weakly self indexing}, so that $\mathcal{L}_j^i = \emptyset$ if $i < j$ (cf. \cite[Sec. 3]{AustinBraam95}). This assumption is satisfied for the moduli space of rank $2$ Higgs bundles with the critical sets labelled with the convention of Section \ref{subsec:properties-energy-function}. The time $t$ downwards gradient flow of $f$ with initial condition $z \in M$ is denoted by $\varphi(z,t)$.

Given a Morse-Bott-Smale function satisfying the conditions of \cite{AustinBraam95}, any pair of critical sets determines a \emph{Morse resolution} defined using the compactification of unbroken flow lines by adding spaces of broken flow lines. More precisely, let $C_\ell$ and $C_u$ be two critical sets with $f(C_\ell) < f(C_u)$, and recall the definition of the space of unbroken flow lines
\begin{equation*}
\mathcal{L}_\ell^u = \{ z \in M \, \mid \, \lim_{t \rightarrow \infty} \varphi(z, t) \in C_\ell, \lim_{t \rightarrow - \infty} \varphi(z,t) \in C_u \} / \mathbb{R} ,
\end{equation*}
where the action of $\mathbb{R}$ on a flow line is by time translation. When it is necessary to specify the upper critical point $y \in C_u$, the space of flow lines is denoted 
\begin{equation*}
\mathcal{L}_\ell^y = \{ z \in M \, \mid \, \lim_{t \rightarrow \infty} \varphi(z, t) \in C_\ell, \lim_{t \rightarrow - \infty} \varphi(z,t) = y \} / \mathbb{R} .
\end{equation*}
Choose $\varepsilon > 0$ so that there are no critical values in the interval $\left[ f(C_u) - \varepsilon, f(C_u) \right)$. Then each flow line emanating from $C_u$ has a unique point of intersection with the level set $f^{-1} \left( f(C_u) - \varepsilon \right)$. Identifying the sphere bundle $S_u^-$ inside the unstable manifold with a subset of the level set gives a homeomorphism $S_u^- \cong W_{u,0}^- \cap f^{-1} \left( f(C_u) - \varepsilon \right)$, and therefore there is an inclusion 
\begin{equation*}
\mathcal{L}_\ell^u \hookrightarrow S_u^- .
\end{equation*}
Now define $\overline{\mathcal{L}_\ell^u}$ to be the closure of $\mathcal{L}_\ell^u$ inside $S_u^-$. Similarly, for a given $y \in C_u$, $\overline{\mathcal{L}_\ell^y}$ denotes the closure of $\mathcal{L}_\ell^y \subset S_y^-$.

\subsubsection{Compactification by broken flow lines}

Before defining the Morse resolution (see Definition \ref{def:Morse-resolution} below), first recall the compactification $\widetilde{\mathcal{L}_\ell^u}$ of the space $\mathcal{L}_\ell^u$ of flow lines defined by adding broken flow lines. The motivation for this definition is that these compactified spaces are used to show that the differentials $\delta$ in the Morse complex satisfy the conditions $\delta \circ \delta = 0$ (cf. \cite[Prop. 3.5]{AustinBraam95}) and that they are compatible with the cup product via the chain relation \cite[(2.2)]{AustinBraam95}. This compactification is explained in detail by Austin and Braam \cite[Sec. 2]{AustinBraam95}.

Each point $x \in \widetilde{\mathcal{L}_\ell^u}$ determines a sequence of intermediate critical sets $C_{m_1}, C_{m_2}, \ldots, C_{m_x} = C_\ell$ with indices $u > m_1 > m_2 > \cdots > m_{n_x} = \ell$ and flow lines between these critical sets
\begin{equation}\label{eqn:broken-flow-line}
x_1 \in \mathcal{L}_{m_1}^u, x_2 \in \mathcal{L}_{m_2}^{m_1}, \ldots, x_{n_x} \in \mathcal{L}_\ell^{m_{n_x-1}} .
\end{equation}
In the sequel, a flow line of this form will be denoted $x = \{ x_1, \ldots, x_{n_x} \} \in \widetilde{\mathcal{L}_\ell^u}$.

\begin{definition}\label{def:Morse-resolution}
The \emph{Morse resolution $\widetilde{\mathcal{L}_\ell^u} \rightarrow \overline{\mathcal{L}_\ell^u}$} is the projection taking a broken flow line $x = \{ x_1, \ldots, x_{n_x} \} \in \widetilde{\mathcal{L}_\ell^u}$ to the first flow line emanating from $C_u$
\begin{align}
\begin{split}
P_{Morse} : \widetilde{\mathcal{L}_\ell^u} & \rightarrow \overline{\mathcal{L}_\ell^u} \\
x = \{ x_1, \ldots, x_{n_x} \} & \mapsto x_1 .
\end{split}
\end{align}
When the upper critical point $y \in C_u$ is fixed, then the restriction of the Morse resolution is denoted
\begin{equation*}
P_{Morse} : \widetilde{\mathcal{L}_\ell^y} \rightarrow \overline{\mathcal{L}_\ell^y}
\end{equation*}
\end{definition}

The following lemma shows that each fibre of $P_{Morse}$ is itself a Morse resolution at a lower critical set.

\begin{lemma}\label{lem:Morse-fibres}
Let $\ell < m < u$, let $P_{Morse} : \widetilde{\mathcal{L}_\ell^u} \rightarrow \overline{\mathcal{L}}_\ell^u$ be the resolution from Definition \ref{def:Morse-resolution} and let $x_1 \in \mathcal{L}_m^u \cap \overline{\mathcal{L}_\ell^u} \subset \overline{\mathcal{L}_\ell^u}$ with $y_1 \in C_m$ the corresponding critical point at the lower limit of the flow line $x_1$. Then 
\begin{equation*}
P_{Morse}^{-1}(x_1) = \widetilde{\mathcal{L}_\ell^{y_1}} .
\end{equation*}
\end{lemma}

\begin{proof}
The fibre $P_{Morse}^{-1}(x_1)$ consists of all broken flow lines connecting $y_1 \in C_m$ to $C_\ell$, which is precisely the resolution $\widetilde{\mathcal{L}_\ell^{y_1}}$.
\end{proof}

\subsection{Resolution of secant varieties}\label{subsec:secant-resolution}

In this section we recall the resolution of secant varieties defined by Bertram \cite{Bertram92} and then prove Theorem \ref{thm:Morse-resolution}, which relates this to the compactification by broken flow lines of the previous section.

%{\bf 2. Explain the resolution of secant varieties that appears in Bertram's thesis.}

Let $L \rightarrow X$ be a line bundle with $\deg L < 0$. Using Schwarzenberger's secant bundle construction \cite{Schwarzenberger64}, Bertram \cite{Bertram92} constructs a resolution of $\Sec_N^L(X) \subset \mathbb{P}H^1(L)$ by repeatedly blowing up $\mathbb{P}H^1(L)$ along the secant varieties of lower dimension. The precise statement we need is from \cite[Sec. 2]{Bertram92}, which is summarised in Lemma \ref{lem:secant-fibres} below, however first we recall the key parts of the construction (see also \cite[Sec. 3.3]{EinNiuPark20} for further explanation).

If $L^* \otimes K$ separates at least $k+1$ points, then let $B^k(L)$ be the secant bundle of $k$ planes associated to the line bundle $L^* \otimes K$ (cf. \cite[Sec. 1]{Bertram92}). Using $H^1(L) \cong H^0(L^* \otimes K)^*$, let $bl_1(\mathbb{P}H^1(L))$ denote the blowup of $\mathbb{P}H^1(L)$ along $X \subset \mathbb{P}H^1(L)$. Each secant variety is the image of $\beta_k : B^k(L) \rightarrow \mathbb{P}H^1(L)$, and $bl_1(B^k(L))$ is defined to be the blowup of $B^k(L)$ along $\beta_k^{-1}(X)$. Let $bl_1(\beta_k)$ be the unique lift of $\beta_k$ to a map
\begin{equation*}
\begin{tikzcd}
bl_1(B^k(L)) \arrow{r}{bl_1(\beta_k)} \arrow{d} & bl_1(\mathbb{P}H^1(L)) \arrow{d} \\
B^k(L) \arrow{r}{\beta_k} &  \mathbb{P}H^1(L) .
\end{tikzcd}
\end{equation*}
Now inductively continue the process. If $bl_n(\mathbb{P}H^1(L))$, $bl_n(B^k(L))$ and $bl_n(\beta_k)$ are defined for all $k \geq n$ and $bl_n(\beta_n)$ is injective, then (after identifying $bl_n(B^n(L))$ with its image), define 
\begin{enumerate}

\item $bl_{n+1}(\mathbb{P}H^1(L))$ to be the blowup of $bl_n(\mathbb{P}H^1(L))$ along $bl_n(B^n(L))$,

\item $bl_{n+1}(B^k(L))$ to be the blowup of $bl_n(B^k(L))$ along $bl_n(\beta_k)^{-1}(bl_n(B^n(L)))$, and

\item $bl_{n+1}(\beta_k)$ to be the unique lift of $bl_n(\beta_k)$ to a map 
\begin{equation*}
\begin{tikzcd}[column sep=5em]
bl_{n+1}(B^k(L)) \arrow{r}{bl_{n+1}(\beta_k)} \arrow{d} & bl_{n+1}(\mathbb{P}H^1(L)) \arrow{d} \\
bl_{n}(B^k(L)) \arrow{r}{bl_{n}(\beta_k)} & bl_{n}(\mathbb{P}H^1(L))
\end{tikzcd}
\end{equation*}

\end{enumerate}

There is a canonical projection map $bl_{n+1}(\mathbb{P}H^1(L)) \rightarrow \mathbb{P} H^1(L)$ given by composing the projections $bl_{n+1}(\mathbb{P}H^1(L)) \rightarrow bl_n(\mathbb{P}H^1(L))$ at each stage of this process. This construction hinges on the injectivity of $bl_n(\beta_n)$ at each step, which is proved in \cite[Prop. 2.3]{Bertram92}. 

Recall from Definition \ref{def:secant-bundle} that the global secant variety $\mathcal{P}_\ell^u$ (resp. the closure $\overline{\mathcal{P}_\ell^u}$) is defined as a fibre bundle over the critical sets, where the fibre over $[L_1 \oplus L_2, \phi]$ is $\Sec_{u-\ell, 0}^{L_1^*L_2}(X) \hookrightarrow \mathbb{P}H^1(L_1^*L_2)$ (resp. $\Sec_{u-\ell}^{L_1^*L_2}(X) \hookrightarrow \mathbb{P}H^1(L_1^*L_2)$). Applying the above construction to the bundle $\mathbb{P} W_u^-$ with fibres $\mathbb{P}H^1(L_1^*L_2)$ gives us the following definition.

\begin{definition}[Resolution of global secant variety]\label{def:global-secant-resolution}
The resolution
\begin{equation*}
P_{Sec}^{\ell,u} : \widetilde{\mathcal{P}_\ell^u} \rightarrow \overline{\mathcal{P}_\ell^u} 
\end{equation*}
is the $(u-\ell-1)^{th}$ blowup of $\overline{\mathcal{P}_\ell^u}$ with fibre over $[L_1\oplus L_2, \phi] \in C_u$ given by $bl_{u-\ell-1}(B^{u-\ell}(L_1^*L_2))$.
\end{definition}

\begin{remark}
Note that this resolution is only nontrivial when $u-\ell > 1$; for example, if $u-\ell = 1$ so that there are no intermediate critical sets between $C_u$ and $C_\ell$, then $\overline{\mathcal{P}_\ell^u} \cong \mathcal{P}_\ell^u$ is the fibre bundle over $C_u$ with fibres $X \subset \mathbb{P}H^1(L_1^*L_2)$, and so $\widetilde{\mathcal{P}_\ell^u} \cong \overline{\mathcal{P}_\ell^u}$.
\end{remark}

The superscript in $P_{Sec}^{\ell,u}$ will be dropped if themeaning is clear from the context. Now we can restate \cite[Cor. 2.5(b)]{Bertram92} in the form needed in the sequel.

\begin{lemma}\label{lem:secant-fibres}
Let $\ell < m < u$, let $P_{Sec} : \widetilde{\mathcal{P}_\ell^u} \rightarrow \overline{\mathcal{P}_\ell^u}$  be the resolution from \cite{Bertram92} and let $x_1 \in \mathcal{P}_m^u \subset \overline{\mathcal{P}_\ell^u}$ be a secant plane in $\mathbb{P}H^1(L_1^*L_2)$ corresponding to a divisor $D$ on $X$.  Then 
\begin{equation*}
P_{Sec}^{-1}(x_1) \cong bl_{u-\ell-1-\deg D}(\mathbb{P}H^1(L_1^*L_2[2D])) .
\end{equation*}
\end{lemma}

Inductively applying Lemma \ref{lem:secant-fibres} to the fibres of 
\begin{equation*}
P_{Sec} : bl_{u-\ell-1-\deg D}(H^1(L_1^*L_2[2D])) \rightarrow \mathbb{P} H^1(L_1^*L_2[2D])
\end{equation*}
shows that each point $s \in bl_{u-\ell-1}(\mathbb{P} H^1(L_1^*L_2))$ corresponds to a sequence of points in secant planes
\begin{align}\label{eqn:broken-secant-plane}
\begin{split}
s_1 \in \Sec_{k_1, 0}^{L_1^*L_2}  \subset \mathbb{P}H^1(L_1^*L_2) & \quad \text{corresponding to a divisor $D_1$ of degree $k_1$} \\
s_2 \in \Sec_{k_2, 0}^{L_1^*L_2[2D_1]} \subset \mathbb{P}H^1(L_1^*L_2[2D_1]) & \quad \text{corresponding to a divisor $D_2$ of degree $k_2$} \\
s_3 \in \Sec_{k_3, 0}^{L_1^*L_2[2D_1+2D_2]} \subset \mathbb{P}H^1(L_1^*L_2[2D_1+2D_2]) & \quad \text{corresponding to a divisor $D_3$ of degree $k_3$} \\
\vdots \quad & \quad \vdots \\
s_{n_s} \in \Sec_{k_s, 0}^{L_1^*L_2[2D]} \subset \mathbb{P}H^1(L_1^*L_2[2D]) & \quad \text{corresponding to a divisor $D_{n_s}$ of degree $k_s$} ,
\end{split}
\end{align}
where $D = D_1 + \cdots + D_{n_s-1}$ is used to simplify the notation in the last line. In the sequel, $s \in bl_{u-\ell-1}(\mathbb{P}H^1(L_1^*L_2))$ of the above form will be denoted $s = \{ s_1, \ldots, s_{n_s} \} \in bl_{u-\ell-1} (\mathbb{P}H^1(L_1^*L_2))$.

%{\bf 3. Relate the resolution to the compactification of flow lines on the moduli space.}

We can now define a map relating the Morse resolution and the resolution of secant varieties. Recall the circle bundle $g : \mathcal{L}_\ell^u \rightarrow \mathcal{P}_\ell^u$ from Theorem \ref{thm:space-unbroken-flow-lines} that takes an unbroken flow line $x \in \mathcal{L}_\ell^y$ to the corresponding point in $\Sec_{u-\ell, 0}^{L_1^*L_2} \subset \mathbb{P}H^1(L_1^*L_2)$. The following definition extends this map to the space of broken flow lines.

\begin{definition}\label{def:resolution-map}
Let $C_\ell$ and $C_u$ be critical sets with $\ell < u$ and let $y \in C_u$. Define $G : \widetilde{\mathcal{L}_\ell^y} \rightarrow bl_{u-\ell-1}(\mathbb{P}H^1(L_1^*L_2))$ by
\begin{equation*}
G(\{x_1, \ldots, x_n\}) = \{ g(x_1), \ldots, g(x_n) \} .
\end{equation*}
\end{definition}

Since the fibres of $g$ are orbits of the circle action, then we have the following result about the fibres of $G$.

\begin{lemma}
Let $[L_1 \oplus L_2, \phi] \in C_u$ be a critical point and let $\{s_1, \ldots, s_n\} \in bl_{u-\ell-1}(\mathbb{P}H^1(L_1^*L_2))$. Then $G^{-1}(\{s_1, \ldots, s_n\}) \cong (S^1)^n$.
\end{lemma}

\begin{proof}
By definition, we have
\begin{equation*}
G^{-1}(\{s_1, \ldots, s_n\}) = \{ (x_1, \ldots, x_n) \in \widetilde{\mathcal{L}_\ell^y} \, \mid \, g(x_i) = s_i,\,  i = 1, \ldots, n \} ,
\end{equation*}
and so $G^{-1}(\{s_1, \ldots, s_n\})$ is a Cartesian product 
\begin{equation*}
g^{-1}(s_1) \times \cdots \times g^{-1}(s_n) \cong (S^1)^n . \qedhere
\end{equation*}
\end{proof}

In particular, we see that the subset of broken flow lines with $n-1$ intermediate critical points has a canonical $(S^1)^n$ action induced from the $S^1$ action on $\mathcal{M}_{Higgs}^{ss}(E)$.

Now we can prove Theorem \ref{thm:Morse-resolution}, which relates the compactification by broken flow lines to the resolution of secant varieties.

\begin{proof}[Proof of Theorem \ref{thm:Morse-resolution}]
Proving that the diagram \eqref{eqn:relate-resolutions} commutes reduces to simply writing down the maps using the above definitions. We have
\begin{equation*}
P_{Sec} \circ G(\{x_1, \ldots, x_n\}) = P_{Sec}(\{ g(x_1), \ldots, g(x_n) \}) = g(x_1)
\end{equation*}
and
\begin{equation*}
g \circ P_{Morse}(\{x_1, \ldots, x_n\}) = g(x_1) ,
\end{equation*}
which completes the proof.
\end{proof}

%\section*{Declaration} 

%{\bf Data availability statement.} Data sharing not applicable to this article as no datasets were generated or analysed during the current study.

%%%%%%%%%%%%%%%%%%%%%%%%%%%%%%%%%%%%%%%%%%%%%%

\end{document}